\documentclass[a4paper, 12pt]{article}

\usepackage[sort&compress]{natbib}
\bibpunct{(}{)}{;}{a}{}{,} 

\usepackage{amsthm, amsmath, amssymb, mathrsfs, multirow, url, subfigure}
\usepackage{graphicx} 
\usepackage{ifthen} 
\usepackage{amsfonts}
\usepackage[usenames]{color}
\usepackage{fullpage}

\numberwithin{equation}{section} 

\theoremstyle{plain} 
\newtheorem{thm}{\indent Theorem}
\newtheorem{cor}{\indent Corollary}
\newtheorem{prop}{\indent Proposition}
\newtheorem{lem}{\indent Lemma}

\theoremstyle{definition}
\newtheorem{defn}{\indent Definition}

\theoremstyle{remark} 
\newtheorem*{gex}{\indent Gaussian Example}
\newtheorem*{eex}{\indent Exponential Example}
\newtheorem*{pex}{\indent Poisson Example}
\newtheorem*{astep}{\indent A-step}
\newtheorem*{pstep}{\indent P-step}
\newtheorem*{cstep}{\indent C-step}

\newcommand{\prob}{\mathsf{P}} 
\newcommand{\E}{\mathsf{E}}

\newcommand{\bel}{\mathsf{bel}}
\newcommand{\pl}{\mathsf{pl}}

\newcommand{\pois}{{\sf Pois}}
\newcommand{\unif}{{\sf Unif}}
\newcommand{\nm}{{\sf N}}
\newcommand{\expo}{{\sf Exp}}
\newcommand{\gam}{{\sf Gamma}}
\newcommand{\stt}{{\sf t}}
\newcommand{\dir}{{\sf Dir}}

\newcommand{\RR}{\mathbb{R}}

\newcommand{\XX}{\mathbb{X}}
\newcommand{\UU}{\mathbb{U}}
\newcommand{\VV}{\mathbb{V}}
\newcommand{\TT}{\mathbb{T}}

\newcommand{\Xbar}{\overline{X}}
\newcommand{\xbar}{\overline{x}}

\renewcommand{\S}{\mathcal{S}}
\renewcommand{\SS}{\mathbb{S}}

\newcommand{\del}{\partial}
\renewcommand{\phi}{\varphi} 
\newcommand{\eps}{\varepsilon}

\title{Inferential models: A framework for prior-free posterior probabilistic inference}
\author{
Ryan Martin \\
Department of Mathematics, Statistics, and Computer Science \\
University of Illinois at Chicago \\
\url{rgmartin@math.uic.edu} \\
\mbox{} \\
Chuanhai Liu \\
Department of Statistics \\
Purdue University \\
\url{chuanhai@stat.purdue.edu}
}
\date{\today}

\begin{document}

\maketitle 

\begin{abstract}
Posterior probabilistic statistical inference without priors is an important but so far elusive goal.  Fisher's fiducial inference, Dempster--Shafer theory of belief functions, and Bayesian inference with default priors are attempts to achieve this goal but, to date, none has given a completely satisfactory picture.  This paper presents a new framework for probabilistic inference, based on \emph{inferential models} (IMs), which not only provides data-dependent probabilistic measures of uncertainty about the unknown parameter, but does so with an automatic long-run frequency calibration property.  The key to this new approach is the identification of an unobservable auxiliary variable associated with observable data and unknown parameter, and the prediction of this auxiliary variable with a random set before conditioning on data.  Here we present a three-step IM construction, and prove a frequency-calibration property of the IM's belief function under mild conditions.  A corresponding optimality theory is developed, which helps to resolve the non-uniqueness issue.  Several examples are presented to illustrate this new approach.  

\smallskip

\end{abstract}

\section{Introduction}
\label{S:intro}

In a statistical inference problem, one attempts to convert \emph{experience}, in the form of observed data, to \emph{knowledge} about the unknown parameter of interest.  The fact that observed data is surely limited implies that there will be some uncertainty in this conversion, and probability is a natural tool to describe this uncertainty.  But a statistical inference problem is different from the classical probability setting because everything---observed data and unknown parameter---is fixed, and so it is unclear where these probabilistic assessments of uncertainty should come from, and how they should be interpreted.  For example, the classical frequentist approach assigns probabilistic assessments of uncertainty (e.g., confidence levels) by considering repeated sampling from the super-population of possible data sets.  These uncertainty measures do not depend on the observed data, so their meaningfulness in a given problem is questionable.  The Bayesian approach, on the other hand, is able to produce meaningful data-dependent probabilistic measures of uncertainty, but the cost is that a prior probability distribution for the unknown parameter is required.  Early efforts to get probabilistic inference without prior specification include Fisher's fiducial inference \citep{zabell1992} and its variants \citep{hannig2009, hannig2012, hannig.lee.2009}, confidence distributions \citep{xie.singh.strawderman.2011, xie.singh.2012}, Fraser's structural inference \citep{fraser1968}, and the Dempster--Shafer theory \citep{dempster2008, shafer1976}.  These methods generate probabilities for inference, but these probabilities may not be easy to interpret, e.g., they may not be properly calibrated across users or experiments.  So recent efforts have focused on incorporating a frequentist element.  In particular, objective Bayes analysis with default/reference priors \citep{mghosh2011, bernardo1979, berger2006, bergerbernardosun2009} attempts to construct priors for which certain posterior inferences, such as credible intervals, closely match that of a frequentist \citep{fraser2011, fraser.reid.marras.yi.2010}.  Calibrated Bayes \citep{rubin1984,dawid1985,little2010} has similar motivations.  But difficulties remain in choosing good reference priors for high-dimensional problems so, despite these efforts, a fully satisfactory framework of objective Bayes inference has yet to emerge.  

The goal of this paper is to develop a new framework for statistical inference, called \emph{inferential models} (IMs).  The seeds for this idea were first planted in \citet{mzl2010} and \citet{zl2010}; here we formalize and extend these ideas towards a cohesive framework for statistical inference.  The jumping off point is a simple association of the observable data $X$ and unknown parameter $\theta \in \Theta$ with an unobservable auxiliary variable $U$.  For example, consider the simple signal plus noise model, $X = \theta + U$, where $U \sim \nm(0,1)$.  If $X=x$ is observed, then we know that $x = \theta + u^\star$, where $u^\star$ is some \emph{unobserved} realization of $U$.  From this it is clear that knowing $u^\star$ is equivalent to knowing $\theta$.  So the IM approach attempts to accurately predict the value $u^\star$ before conditioning on $X=x$.  The benefit of focusing on $u^\star$ rather than $\theta$ is that more information is available about $u^\star$: indeed, all that is known about $\theta$ is that it sits in $\Theta$, while $u^\star$ is known to be a realization of a draw $U$ from an \emph{a priori} distribution, in this case $\nm(0,1)$, that is fully specified by the postulated sampling model.  However, this \emph{a priori} distribution alone is insufficient for accurate prediction of $u^\star$.  Therefore, we adopt a so-called \emph{predictive random set} for predicting $u^\star$, which amounts to a sort of ``smearing'' of this distribution for $U$.  When combined with the association between observed data, parameters, and auxiliary variables, these random sets produce prior-free, data-dependent probabilistic assessments of uncertainty about $\theta$.  

To summarize, an IM starts with an association between data, parameters, and auxiliary variables and a predictive random set, and produces prior-free, post-data probabilistic measures of uncertainty about the unknown parameter.  The following associate-predict-combine steps provide a simple yet formal IM construction.  The details of each of these three steps will be fleshed out in Section~\ref{S:im}.  

\begin{astep}
Associate the unknown parameter $\theta$ to each possible $(x,u)$ pair to obtain a collection of sets $\Theta_x(u)$ of candidate parameter values.  
\end{astep}

\begin{pstep}
Predict $u^\star$ with a valid predictive random set $\S$.  
\end{pstep}

\begin{cstep}
Combine $X=x$, $\Theta_x(u)$, and $\S$ to obtain a random set $\Theta_x(\S) = \bigcup_{u \in \S} \Theta_x(u)$.  Then, for any assertion $A \subseteq \Theta$, compute the probability that the random set $\Theta_x(\S)$ is a subset of $A$ as a measure of the available evidence in $x$ supporting $A$.  
\end{cstep}

The A-step is meant to emphasize the use of unobservable but predictable auxiliary variables in the statistical modeling step.  These auxiliary variables make it possible to introduce posterior probability-like quantities without a prior distribution for $\theta$.  The P-step is new and unique to the inferential model framework.  The key is that $\Theta_x(\S)$ contains the true $\theta$ if and only if $\S$ contains $u^\star$.  Then the validity condition in the P-step ensures that $\S$ will hit its target with large probability which, in turn, guarantees that probabilistic output from the C-step has a desirable frequency-calibration property.  This, together with its dependence on the observed data $x$, makes the IM's probabilistic output meaningful both within and across experiments.  

The remainder of the paper is organized as follows.  Section~\ref{S:im} provides the details of the IM analysis, specifically the three-step construction outlined above, as well as a description of calculation and interpretation of the IM output: a posterior belief function.  These ideas are illustrated with a simple Poisson mean example.  After arguing, in Section~\ref{S:im}, that the IM output provides a meaningful summary of one's uncertainty about $\theta$ after seeing $X=x$, we prove a frequency calibration property of the posterior belief functions in Section~\ref{S:validity} which establishes the meaningfulness of the posterior belief function across different users and experiments.  As a consequence of this frequency-calibration property, we show in Section~\ref{SS:freq} that the IM output can easily be used to design new frequentist decision procedures having the desired control on error probabilities, etc.  Some basic but fundamental results on IM optimality are presented in Section~\ref{S:optimality}.  Section~\ref{S:more-examples} gives IM-based solutions to two non-trivial examples, both involving some sort of marginalization.  Nonetheless, these examples are relatively simple and they illustrate the advantages of the IM approach.  Concluding remarks are given in Section~\ref{S:discuss}, and R codes for the examples are available on the first author's website: \url{www.math.uic.edu/~rgmartin}.

\section{Inferential models}
\label{S:im}

\subsection{Auxiliary variable associations}
\label{SS:association}

If $X$ denotes the observable sample data, then the sampling model is a probability distribution $\prob_{X|\theta}$ on the sample space $\XX$, indexed by a parameter $\theta \in \Theta$.  Here $X$ may consist of a collection of $n$ (possibly vector-valued) data points, in which case both $\prob_{X|\theta}$ and $\XX$ would depend on $n$.  The sampling model for $X$ is induced by an auxiliary variable $U$, for given $\theta$.  Let $\UU$ be an (arbitrary) auxiliary space, equipped with a probability measure $\prob_U$.  In applications, $\UU$ can often be a unit hyper-cube and $\prob_U$ Lebesgue measure.  The sampling model $\prob_{X|\theta}$ shall be determined by the following ``algorithm:''
\begin{equation}
\label{eq:sm}
\text{sample $U \sim \prob_U$ and set $X = a(U,\theta)$},
\end{equation}
for an appropriate mapping $a: \UU \times \Theta \to \XX$.  The key is the association of the observable $X$, the unknown $\theta$, and the auxiliary variable $U$ through the relation $X = a(U,\theta)$.  This particular formulation of the sampling model is not really a restriction.  In fact, the two-step construction of the observable $X$ in \eqref{eq:sm} is often consistent with scientific understanding of the underlying process under investigation; linear models form an interesting class of examples.  As another example, suppose $X = (X_1,\ldots,X_n)$ consists of an independent sample from a continuous distribution.  If the corresponding distribution function $F_\theta$ is invertible, then $a(\theta, U)$ may be written as 
\begin{equation}
\label{eq:iid.amodel}
a(\theta,U) = \bigl( F_\theta^{-1}(U_1),\ldots,F_\theta^{-1}(U_n) \bigr),  
\end{equation}
where $U=(U_1,\ldots,U_n)$ is a set of independent $\unif(0,1)$ random variables.  

The notation $X = a(\theta,U)$ chosen to represent the association between $(X,\theta,U)$ is just for simplicity.  In fact, this association need not be described by a formal equation.  As the Poisson example below shows, all we need is a recipe, like that in \eqref{eq:sm}, describing how to produce a sample $X$, for a given $\theta$, based on a realization $U \sim \prob_U$.  

\begin{gex}
Consider the problem of inference on the mean $\theta$ based on a single sample $X \sim \nm(\theta,1)$.  In this case, the association linking $X$, $\theta$, and an auxiliary variable $U$ may be written as $U = \Phi(X - \theta)$ or, equivalently, $X = \theta + \Phi^{-1}(U)$, where $U \sim \unif(0,1)$, and $\Phi$ is the standard Gaussian distribution function.  
\end{gex}

\begin{pex}
Consider the problem of inference on the mean $\theta$ of a Poisson population based on a single observation $X$.  For this discrete problem, the association for $X$, given $\theta$, may be written as 
\begin{equation}
\label{eq:pois-sm}
F_\theta(X-1) \leq 1-U < F_\theta(X), \quad U \sim \unif(0,1), 
\end{equation}
where $F_\theta$ denotes the $\pois(\theta)$ distribution function.  This representation is familiar for simulating $X \sim \pois(\theta)$, i.e., one can first sample $U \sim \unif(0,1)$ and then choose $X$ so that the inequalities in \eqref{eq:pois-sm} are satisfied.  But here we also interpret \eqref{eq:pois-sm} as a means to link data, parameter, and auxiliary variable.  
\end{pex}

It should not be surprising that, in general, there are many associations for a given sampling model.  In fact, for a given sampling model $\prob_{X|\theta}$, there are as many associations as there are triplets $(\UU,\prob_U,a)$ such that $\prob_{X|\theta}$ equals the push-forward measure $\prob_U a_\theta^{-1}$, with $a_\theta(\cdot) = a(\theta,\cdot)$.  For example, if $X \sim \nm(\theta,1)$, then each of the following defines an association: $X=\theta+U$ with $U \sim \nm(0,1)$, $X=\theta+\Phi^{-1}(U)$ with $U \sim \unif(0,1)$, and 
\[ X = \begin{cases} \theta + U & \text{if $\theta \geq 0$}, \\ \theta - U & \text{if $\theta < 0$}, \end{cases} \quad \text{with $U \sim \nm(0,1)$}. \]
Presently, there appears to be no strong reason to choose one of these associations over the other.  However, the optimality theory presented in Section~\ref{S:optimality} helps to resolve this non-uniqueness issue, that is, the optimal IM depends only on the sampling model, and not on the chosen association.  From a practical point of view, we prefer, for continuous data problems, associations which are continuous in both $\theta$ and $U$, which rules out the latter of the three associations above.  Also, we tend to prefer the representation with a uniform $U$, any other choice being viewed as just a reparametrization of this one.  It will become evident that this view is without loss of generality for simple problems with a one-dimensional auxiliary variable.  The case when $U$ is moderate- to high-dimensional is more challenging and we defer its discussion to Section~\ref{S:discuss}.

\subsection{Three-step IM construction}
\label{SS:three.step}

\subsubsection{Association step}
\label{SSS:astep}

The association \eqref{eq:sm} plays two distinct roles.  Before the experiment, the association characterizes the predictive probabilities of the observable $X$.  But once $X = x$ is observed, the role of the association changes.  The key idea is that the observed $x$ and the unknown $\theta$ must satisfy 
\begin{equation}
\label{eq:aeqn}
x = a(u^\star,\theta)
\end{equation}
for some unobserved realization $u^\star$ of $U$.  Although $u^\star$ is unobserved, there is information available about the nature of this quantity; in particular, we know exactly the distribution $\prob_U$ from which it came.  

Of course, the value of $u^\star$ can never be known, \emph{but if it were}, the inference problem would be simple: given $X=x$, just solve the equation $x = a(u^\star,\theta)$ for $\theta$.  More generally, one could construct the set of solutions $\Theta_x(u^\star)$, where 
\begin{equation}
\label{eq:inverse1}
\Theta_x(u) = \{\theta: x = a(u,\theta)\}, \quad x \in \XX, \quad u \in \UU. 
\end{equation}
For continuous-data problems, $\Theta_x(u)$ is typically a singleton for each $u$; for other problems, it could be a set.  In either case, given $X=x$, $\Theta_x(u^\star)$ represents the best possible inference in the sense that \emph{the true $\theta$ is guaranteed to be in} $\Theta_x(u^\star)$.  

\begin{gex}[cont]
The Gaussian mean problem is continuous, so the association $x = \theta + \Phi^{-1}(u)$ identifies a single $\theta$ for each fixed $(x,u)$ pair.  Therefore, $\Theta_x(u) = \{x-\Phi^{-1}(u)\}$.  In this case, clearly, if $u^\star$ were somehow observed, then the true $\theta$ could be determined with complete certainty.  
\end{gex}

\begin{pex}[cont]
Integration-by-parts reveals that the $\pois(\theta)$ distribution function $F_\theta$ satisfies $F_\theta(x) = 1-G_{x+1}(\theta)$, where $G_a$ is a $\gam(a,1)$ distribution function.  Therefore, from \eqref{eq:pois-sm}, we get the $u$-interval $G_{x+1}(\theta) < u \leq G_x(\theta)$.  Inverting this $u$-interval produces the following $\theta$-interval:
\begin{equation}
\label{eq:pois-aeqn-t-mapping}
 \Theta_x(u) = \bigl( G_x^{-1}(u), G_{x+1}^{-1}(u) \bigr].  
\end{equation}
If $u^\star$ was available, then $\Theta_x(u^\star)$ would provide the best possible inference in the sense that the true value of $\theta$ is guaranteed to sit inside this interval.  But even in this ideal case there is no information available to identify the exact location of $\theta$ in $\Theta_x(u^\star)$.  
\end{pex}

\subsubsection{Prediction step}
\label{SSS:pstep}

The above discussion highlights the importance of the auxiliary variable for inference.  It is, therefore, only natural that the inference problem should focus on accurately predicting the unobserved $u^\star$.  To predict $u^\star$ with a certain desired accuracy, we employ a so-called a \emph{predictive random set}.  First we give the simplest description of a predictive random set and provide a useful example.  More general descriptions will be given later. 

Let $u \mapsto S(u)$ be a mapping from $\UU$ to a collection of $\prob_U$-measurable subsets of $\UU$; one decent example of such a mapping $S$ is given in equation \eqref{eq:default.prs} below.  Then the predictive random set $\S$ is obtained by applying the set-valued mapping $S$ to a draw $U \sim \prob_U$, i.e., $\S = S(U)$ with $U \sim \prob_U$.  The intuition is that if a draw $U \sim \prob_U$ is a good prediction for the unobserved $u^\star$, then the random set $\S = S(U)$ should be even better in the sense that there is high probability that $\S \ni u^\star$.  

\begin{gex}[cont]
In this example we may predict the unobserved $u^\star$ with a predictive random set $\S$ defined by the set-valued mapping 
\begin{equation}
\label{eq:default.prs}
S(u) = \bigl\{ u' \in (0,1): |u' - 0.5| < |u-0.5| \bigr\}, \quad u \in (0,1).  
\end{equation}
As this predictive random set is designed to predict an unobserved uniform variate, we may also employ \eqref{eq:default.prs} in other problems, including the Poisson example.  
\end{gex}

There are, of course, other choices of $S(u)$, e.g., $[0,u)$, $(u,1]$, $(0.5u, 0.5 + 0.5u)$ and more.  Although some other choice of $\S=S(U)$ might perform slightly better depending on the assertion of interest, \eqref{eq:default.prs} seems to be a good default choice, provided that the association satisfies certain monotonicity conditions.  See Sections~\ref{S:validity} and \ref{S:optimality} for more on the choice predictive random sets. 

For the remainder of this paper, we shall mostly omit the set-valued mapping $S$ from the notation and speak directly about the predictive random set $\S$.  That is, the predictive random set $\S$ will be just a random subset of $\UU$ with distribution $\prob_{\S}$.  In the above description, $\prob_{\S}$ is just the push-forward measure $\prob_U S^{-1}$.

\subsubsection{Combination step}
\label{SSS:cstep}

For the time being, let us assume that the predictive random set $\S$ is satisfactory for predicting the unobserved $u^\star$; this is actually easy to arrange, but we defer discussion until Section~\ref{S:validity}.  To transfer the available information about $u^\star$ to the $\theta$-space, our last step is to combine the information in the association, the observed $X=x$, and the predictive random set $\S$.  The intuition is that, if $u^\star \in \S$, then the true $\theta$ must be in the set $\Theta_x(u)$, from \eqref{eq:inverse1}, for at least one $u \in \S$.  So, logically, it makes sense to consider, for inference about $\theta$, the expanded set 
\begin{equation}
\label{eq:inverse2}
\Theta_x(\S) = \textstyle\bigcup_{u \in \S} \Theta_x(u).
\end{equation}
The set $\Theta_x(\S)$ contains those values of $\theta$ which are consistent with the observed data and sampling model for at least one candidate $u^\star$ value $u \in \S$.  Since $\theta \in \Theta_x(\S)$ if and only if the unobserved $u^\star \in \S$, if we are willing to accept that the predictive random set $\S$ is satisfactory for predicting $u^\star$, then $\Theta_x(\S)$ will do equally well at capturing $\theta$.  

Now consider an assertion $A$ about the parameter of interest $\theta$.  Mathematically, an assertion is just a subset of $\Theta$, but it acts much like a hypothesis in the context of classical statistics.  To summarize the evidence in $x$ that supports the assertion $A$, we calculate the probability that $\Theta_x(\S)$ is a subset of $A$, i.e., 
\begin{equation}
\label{eq:lev}
\bel_x(A) = \prob_\S\{\Theta_x(\S) \subseteq A \mid \Theta_x(\S) \neq \varnothing\}.  
\end{equation}
We refer to $\bel_x(A)$ as the \emph{belief function} at $A$.  Naturally, $\bel_x$ also depends on the choice of association and predictive random set, but for now we suppress this dependence in the notation.  There are some similarities between our belief function and that of Dempster--Shafer theory \citep{shafer1976}.  For example, $\bel_x$ is subadditive in the sense that if $A$ is a non-trivial subset of $\Theta$, then $\bel_x(A) + \bel_x(A^c) \leq 1$ with equality if and only if $\Theta_x(\S)$ is a singleton with $\prob_{\S}$-probability~1.  However, our use of the predictive random set (and our emphasis on validity in Section~\ref{S:validity}) separates our approach from that of Dempster--Shafer; see \citet{mzl2010}.  

Here we make two technical remarks about the belief function in \eqref{eq:lev}.  First, in the problems considered in this paper, the case $\Theta_x(\S) = \varnothing$ is a $\prob_\S$-null event, so the belief function can be simplified as $\bel_x(A) = \prob_\S\{\Theta_x(\S) \subseteq A\}$, no conditioning.  This simplification may not hold in problems where the observation $X=x$ can induce a constraint on the auxiliary variable $u$.  For example, consider the Gaussian example from above, but suppose that the mean is known to satisfy $\theta \geq 0$.  In this case, it is easy to check that $\Theta_x(\S) = \varnothing$ iff $\Phi^{-1}(\inf \S) > x$, an event which generally has positive $\prob_\S$-probability.  So, in general, we can ignore conditioning provided that 
\begin{equation}
\label{eq:constraint}
\Theta_x(u) \neq \varnothing \quad \text{for all $x$ and all $u$}.
\end{equation}
The IM framework can be modified in cases where \eqref{eq:constraint} fails, but we will not discuss this here; see \citet{leafliu2012}.  Second, measurability of $\bel_x(A)$, as a function of $x$ for given $A$, which is important in what follows, is not immediately clear from the definition and should be assessed case-by-case.  However, in our experience and in all examples herein, $\bel_x(A)$ is a nice measurable function of $x$.

Unlike with an ordinary additive probability measure, to reach conclusions about $A$ based on $\bel_x$ one must know \emph{both} $\bel_x(A)$ and $\bel_x(A^c)$; for example, in the extreme case of ``total ignorance'' about $A$, one has $\bel_x(A)=\bel_x(A^c) = 0$.  It is often more convenient to work with a different but related function 
\begin{equation}
\label{eq:uev}
\pl_x(A) = 1-\bel_x(A^c) = \prob_\S\{\Theta_x(\S) \not\subseteq A^c \mid \Theta_x(\S) \neq \varnothing\}, 
\end{equation}
called the \emph{plausibility function} at $A$; when $A=\{\theta\}$ is a singleton, we write $\pl_x(\theta)$ instead of $\pl_x(\{\theta\})$.  From the subadditivity of the belief function, it follows that $\bel_x(A) \leq \pl_x(A)$ for all $A$.  In what follows, to summarize the evidence in $x$ supporting $A$, we shall report the pair $\bel_x(A)$ and $\pl_x(A)$, also known as lower and upper probabilities. 

\begin{gex}[cont]
With the predictive random set $\S$ in \eqref{eq:default.prs}, the random set $\Theta_x(\S)$ is given by
\begin{align*}
\Theta_x(\S) & = \textstyle\bigcup_{u \in \S} \{x - \Phi^{-1}(u)\} \\
& = \bigl( x - \Phi^{-1}(0.5 + |U-0.5|), \, x - \Phi^{-1}(0.5 - |U-0.5|) \bigr) \\
& = \bigl( \underline{\Theta}_x(U),\, \overline{\Theta}_x(U) \bigr), \quad \text{say},
\end{align*}
where $U \sim \unif(0,1)$.  For a singleton assertion $A = \{\theta\}$, it is easy to see that the belief function is zero.  But the plausibility function is 
\begin{align}
\pl_x(\theta) & = 1 - \prob_U\bigl\{\underline{\Theta}_x(U) > \theta\bigr\} - \prob_U\bigl\{\overline{\Theta}_x(U) < \theta\bigr\} \notag \\
& = 1 - |2\Phi(x-\theta) - 1|. \label{eq:gauss.pl}
\end{align}
A plot of $\pl_x(\theta)$, with $x=5$, as a function of $\theta$, is shown in Figure~\ref{fig:pois1}(a).  The symmetry around the observed $x$ is apparent, and all those $\theta$ values in a neighborhood of $x=5$ are relatively plausible.  See Section~\ref{SS:freq} for more statistical applications of this graph.  
\end{gex} 

\begin{pex}[cont]
With the same predictive random set as in the previous example, the random set $\Theta_x(\S)$ is given by 
\begin{align*}
\Theta_x(\S) & = \textstyle\bigcup_{u \in \S} \bigl( G_x^{-1}(u), G_{x+1}^{-1}(u)\bigr] \\
& = \bigl( G_x^{-1}\bigl( 0.5 - |U-0.5| \bigr), G_{x+1}^{-1}\bigl( 0.5 + |U-0.5| \bigr) \bigr) \\
& = \bigl( \underline{\Theta}_x(U),\, \overline{\Theta}_x(U) \bigr), \quad \text{say},
\end{align*}
where $U$ is a random draw from $\unif(0,1)$.  For a singleton assertion $A = \{\theta\}$, again the belief function is zero, but the plausibility function is 
\begin{align}
\pl_x(\theta) & = 1 - \prob_U\bigl\{\underline{\Theta}_x(U) > \theta\bigr\} - \prob_U\bigl\{\overline{\Theta}_x(U) < \theta\bigr\} \notag \\
& = 1 - \max\{1-2G_x(\theta), 0\} - \max\{2G_{x+1}(\theta)-1,0\}. \label{eq:pois-evid}
\end{align}
A graph of $\pl_x(\theta)$, with $x=5$, as a function of $\theta$ is shown in Figure~\ref{fig:pois1}(b).  The plateau indicates that no $\theta$ values in a neighborhood of 5 can be ruled out.  Like in the Gaussian example, $\theta$ values in an interval around 5 are all relatively plausible.  

\citet{dempster2008} gives a different analysis of this same Poisson problem.  His plausibility function for the singleton assertion $A=\{\theta\}$ is $r_x(\theta) = e^{-\theta} \theta^x/x!$, which is the Poisson mass function treated as a function of $\theta$.  This function has a similar shape to that in Figure~\ref{fig:pois1}(b), but the scale is much smaller.  For example, $r_5(5) = 0.175$, suggesting that the assertion $\{\theta=5\}$ is relatively implausible, even though $X=5$ was observed.  Compare this to $\pl_5(5)=1$.  We would argue that, if $X=5$ is observed, then no plausibility function threshold should be able to rule out $\{\theta=5\}$; in that case, $\pl_5(5) = 1$ makes more sense.  Furthermore, as Dempster's analysis is similar to ours but with an invalid predictive random set, namely, $\S = \{U\}$, with $U \sim \unif(0,1)$, the corresponding plausibility function is not properly calibrated for all assertions.  
\end{pex}

\begin{figure}
\begin{center}
\subfigure[Gaussian example]{\scalebox{0.6}{\includegraphics{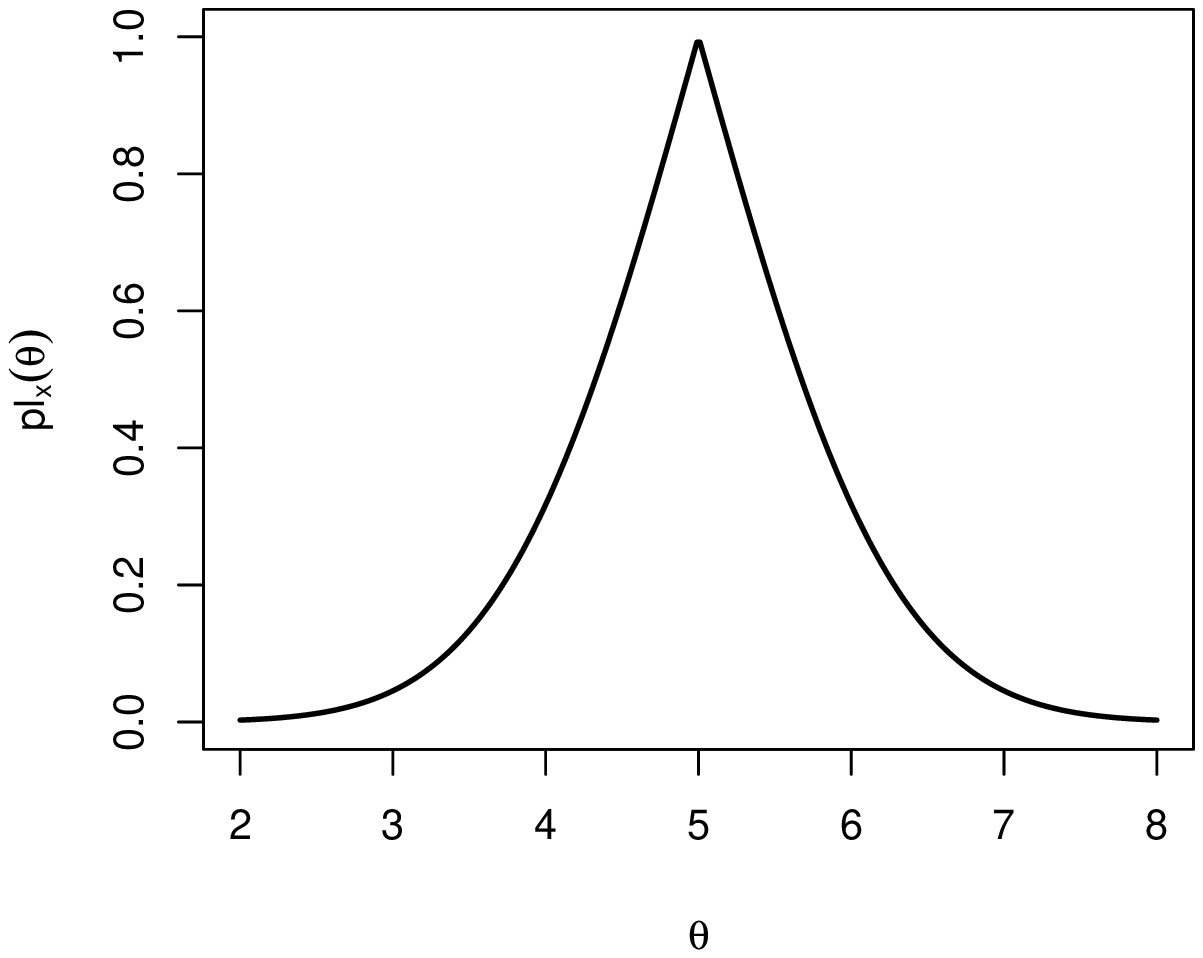}}}
\subfigure[Poisson example]{\scalebox{0.6}{\includegraphics{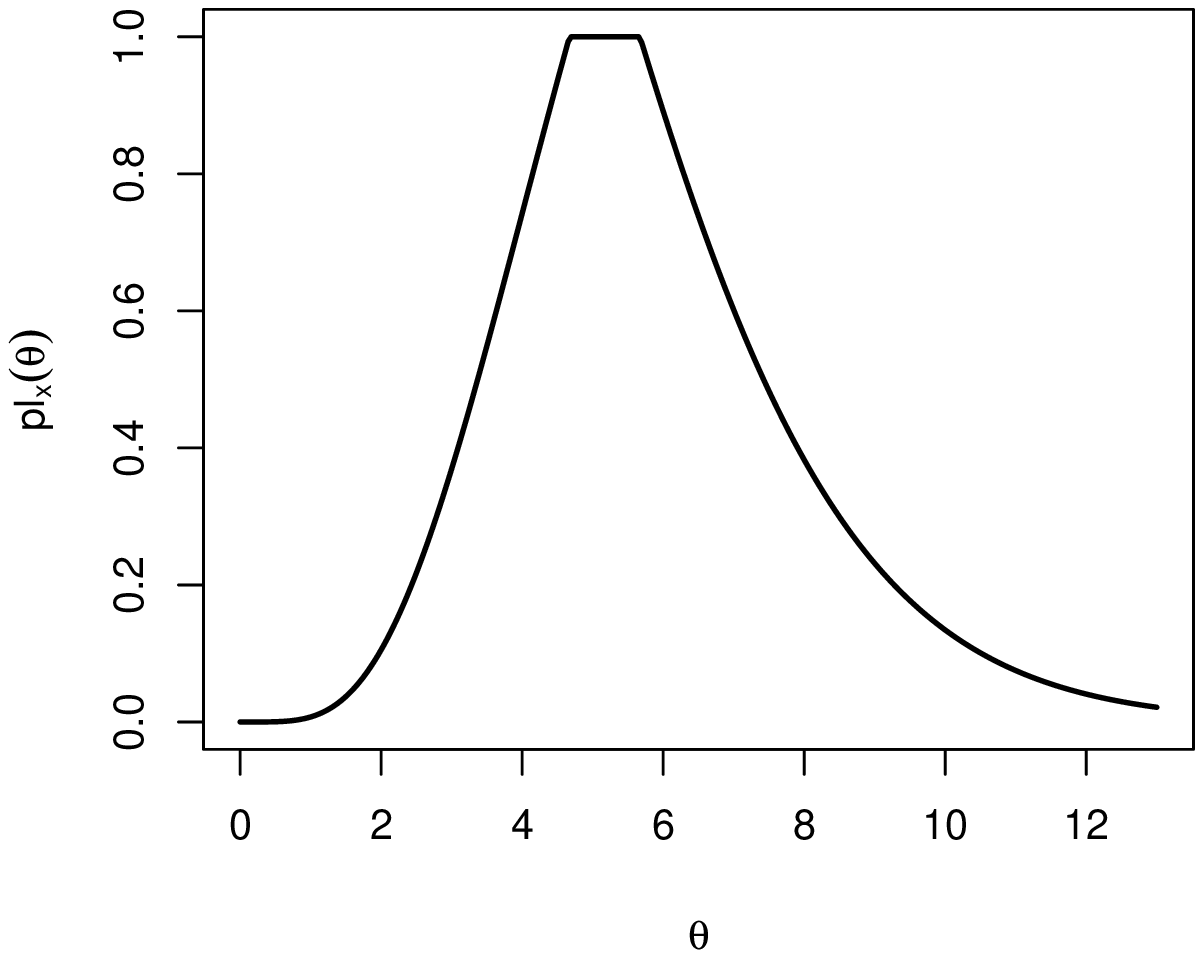}}}
\caption{Plot of the plausibility functions $\pl_x(\theta)$, as functions of $\theta$, in (a) the Gaussian example and (b) the Poisson example.  In both cases, $X=5$ is observed.}
\label{fig:pois1}
\end{center}
\end{figure}

\subsection{Interpretation of the belief function}
\label{SS:interpretation}

It is clear that the belief function depends on the observed $x$ and so must be meaningful within the problem at hand.  But while it is data-dependent, $\bel_x(A)$ is not a posterior probability for $A$ in the familiar Bayesian sense.  In fact, under our assumption that $\theta$ is fixed and non-random, there can be no non-trivial posterior distribution on $\Theta$.  The way around this limitation is to drop the requirement that posterior inference be based on a bona fide probability measure \citep[e.g.,][]{wasserman1990, walley1996, heathsudderth1978}.  Therefore, we recommend interpreting $\bel_x(A)$ and $\bel_x(A^c)$ as degrees of belief, rather than ordinary probabilities, even though they manifest from $\prob_U$-probability calculations.  More precisely, $\bel_x(A)$ and $\bel_x(A^c)$ represent the knowledge gained about the respective claims $\theta \in A$ and $\theta \not\in A$ based on both the observed $x$ and prediction of the auxiliary variable.  

\subsection{Summary}
\label{SS:summary}

The familiar sampling model appears in the A-step, but it is the corresponding association which is of primary importance.  This association, in turn, determines the auxiliary variable which is to be the focus of the IM framework.  We propose to predict the unobserved value of this auxiliary variable in the P-step with a predictive random set $\S$, which is chosen to have certain desirable properties (see Definition~\ref{def:credible} below).  This use of a predictive random set is likely the aspect of the IM framework which is most difficult to swallow, but the intuition should be clear: one cannot hope to accurately predict a fixed value $u^\star$ by an ordinary continuous random variable.  With the association, predictive random set, and observed $X=x$ in hand, one proceeds to the C-step where a random set $\Theta_x(\S)$ on the parameter space is obtained.  As this random set corresponds to a set of ``reasonable'' $\theta$ values, given $x$, it is natural to summarize the support of an assertion $A$ by the probability that $\Theta_x(\S)$ is a subset of $A$.  This probability is exactly the belief function that characterizes the output of the IM and an argument is presented that justifies the meaningfulness of $\bel_x(A)$ and $\pl_x(A)$ as summaries of the evidence in favor of $A$.  

Finally, we mention that the predictive random set $\S$ can depend on the assertion $A$ of interest.  That is, one might consider using one predictive random set, say $\S_A$, to evaluate $\bel_x(A)$, and another predictive random set, say $\S_{A^c}$, to evaluate $\pl_x(A) = 1-\bel_x(A^c)$.  In Section~\ref{S:optimality} we show that this is actually a desirable strategy, in the sense that the optimal predictive random set depends on the assertion in question.  In what follows, this dependence of the predictive random set on the assertion will be kept implicit.

\section{Theoretical validity of IMs}
\label{S:validity}

\subsection{Intuition}
\label{SS:validity.intuition}

In Section~\ref{S:im} we argued that $\bel_x(A;\S)$ and $\pl_x(A;\S)$ together provide a meaningful summary of evidence in favor of $A$ for the given $X=x$; our notation now explicitly indicates the dependence of these function on the predictive random set $\S$.  In this section we show that $\bel_X(A;\S)$ and $\pl_X(A;\S)$ are also meaningful as functions of the random variable $X \sim \prob_{X|\theta}$ for a fixed assertion $A$.  For example, we show that $\bel_X(A)$ is frequency-calibrated in the following sense: if $\theta \not\in A$, then $\prob_{X|\theta}\{\bel_X(A;\S) \geq 1-\alpha\} \leq \alpha$ for each $\alpha \in [0,1]$.  In other words, the amount of evidence in favor of a false $A$ can be large with only small probability.  This property means that the belief function is appropriately scaled for objective scientific inference.  A similar property also holds for $\pl_X(A)$.  We refer to this frequency-calibration property as \emph{validity} (Definition~\ref{def:validity}).  \citet{bernardo1979}, \citet{rubin1984} and \citet{dawid1985} give similar investigations of frequency-calibration and of objective priors for Bayesian inference.

\subsection{Predictive random sets}
\label{SS:credible}

We start with a few definitions, similar to those found in \citet{mzl2010} and \citet{zl2010}.  Set $Q_{\S}(u) = \prob_{\S}\{ \S \not\ni u\}$, for $u \in \UU$, which is the probability that the predictive random set $\S$ misses the specified target $u$.  Ideally, $\S$ will be such that the random variable $Q_{\S}(U)$, a function of $U \sim \prob_U$, will be probabilistically small.  This suggests a connection between $\prob_\S$ and $\prob_U$ which will be made precise in Theorem~\ref{thm:nested.prs}.  

\begin{defn}
\label{def:credible}
A predictive random set $\S$ is \emph{valid} for predicting the unobserved auxiliary variable if $Q_\S(U)$, as a function of $U \sim \prob_U$, is stochastically no larger than $\unif(0,1)$, i.e., for each $\alpha \in (0,1)$, $\prob_U\{Q_{\S}(U) \geq 1-\alpha\} \leq \alpha$.  If ``$\leq \alpha$'' can be replaced by ``$=\alpha$,'' then $\S$ is \emph{efficient}.  
\end{defn}

In words, validity of $\S$ implies that the probability that it misses a target $u$ is large for only a small $\prob_U$-proportion of possible $u$ values.  The predictive random set $\S$ defined by the mapping \eqref{eq:default.prs} is both valid and efficient.  Indeed, it is easy to check that, in this case, $Q_{\S}(u) = |2u-1|$.  Therefore, if $U \sim \unif(0,1)$ then $Q_{\S}(U) \sim \unif(0,1)$ too.  Corollary~\ref{thm:prs} below gives a simple and general recipe for constructing a valid and efficient $\S$.  

There is an important and apparently fundamental concept related to validity of predictive random sets, namely, \emph{nesting}.  We say that a collection of sets $\SS \subseteq 2^{\UU}$ is nested if, for any pair of sets $S$ and $S'$ in $\SS$, we have $S \subseteq S'$ or $S' \subseteq S$.  We shall also implicitly assume, without loss of generality, that $\prob_U(S) > 0$ for some $S \in \SS$; the user can easily arrange this.  The following theorem shows that if the predictive random set $\S$ is nested, i.e., if $\S$ is supported on a nested collection of sets $\SS$, then it is valid.  

\begin{thm}
\label{thm:nested.prs}
Let $\SS \subseteq 2^{\UU}$ be a nested collection of $\prob_U$-measurable subsets $S$ of $\UU$.  
Define a predictive random set $\S$ with distribution $\prob_\S$, supported on $\SS$, such that
\[ \prob_\S\{\S \subseteq K\} = \sup_{S \in \SS: S \subseteq K} \overline\prob_U(S), \quad K \subseteq \UU, \]
where $\overline\prob_U(\cdot) = \prob_U(\cdot) / \sup_{S\in\SS} \prob_U(S)$.  Then $\S$ is valid in the sense of Definition~\ref{def:credible}.  
\end{thm}

\begin{proof}[\indent Proof]
The idea of the proof is that $Q_\S(u)=\prob_\S\{\S \not\ni u\}$ is large iff $u$ sits outside a set that contains most realizations of $\S$.  To make this formal, take any $\alpha \in (0,1)$ and let $S_\alpha = \bigcap\{S \in \SS: \prob_U(S) \geq 1-\alpha\}$ be the smallest set in $\SS$ with $\prob_U$-probability no less than $1-\alpha$; here, the intersection over an empty collection of sets is taken to be $\UU$.  Since $\SS$ is nested, $S_\alpha \in \SS$, $\prob_U(S_\alpha) \geq 1-\alpha$, and 
\[ \prob_\S\{\S \subseteq S_\alpha\} = \sup_{S \in \SS: S \subseteq S_\alpha} \overline\prob_U(S) = \overline\prob_U(S_\alpha) \geq \prob_U(S_\alpha) \geq 1-\alpha. \]
Therefore, since $Q_\S(u) > 1-\alpha$ iff $u \not\in S_\alpha$, we get $\prob_U\{Q_\S(U) > 1-\alpha\} = \prob_U(S_\alpha^c) = 1-\prob_U(S_\alpha) \leq \alpha$.  Finally, validity follows since $\alpha$ was arbitrary.
\end{proof}

It is clear that, if $\prob_U$ is absolutely continuous and the nested support $\SS$ is sufficiently rich, then the predictive random set defined above is also efficient.  Specifically, if $\UU \in \SS$ and, for $S_\alpha$ defined in the proof above, $\prob_U(S_\alpha) = 1-\alpha$ for every $\alpha \in (0,1)$.  This vague argument for efficiency is made more precise in the next important special case.   

\begin{cor}
\label{thm:prs}
Suppose the $\prob_U$ is non-atomic, and let $h$ be a real-valued function on $\UU$.  Then the predictive random set $\S = \{u \in \UU: h(u) < h(U)\}$, with $U \sim \prob_U$, is valid.  If $h$ is continuous and constant only on $\prob_U$-null sets, then it is also efficient.    
\end{cor}

\begin{proof}[\indent Proof]
Validity is a consequence of Theorem~\ref{thm:nested.prs} and the fact that this $\S$ is nested.  To prove the efficiency claim, let $H$ be the distribution function of $h(U)$ when $U \sim \prob_U$.  Then, for $u \in \UU$, $Q_\S(u) = \prob_U\{h(U) \leq h(u)\} = H(h(u))$.  If $h$ satisfies the stated conditions, then $h(U)$ is a continuous random variable.  Therefore, if $U \sim \prob_U$, then $Q_\S(U) = H(h(U)) \sim \unif(0,1)$, so efficiency follows.  
\end{proof}

The above results demonstrate that nesting is a sufficient condition for predictive random set validity.  But nesting is not a necessary condition \citep{mzl2010}.  The real issue, however, is the performance of the corresponding IM.  We show in Section~\ref{S:optimality} that for any non-nested predictive random set $\S$, there is a nested predictive random set $\S'$ such that the IM based on $\S'$ is ``at least as good'' as that based on $\S$.

\subsection{IM validity}
\label{SS:validity}

Validity of the underlying predictive random set $\S$ is essentially all that is needed to prove the meaningfulness of the corresponding IM/belief function.  Here meaningfulness refers to a calibration property of the belief function.  

\begin{defn}
\label{def:validity}
Suppose $X \sim \prob_{X|\theta}$ and let $A$ be an assertion of interest.  Then the IM with belief function $\bel_x$ is \emph{valid for $A$} if, for each $\alpha \in (0,1)$, 
\begin{equation}
\label{eq:cred.lev}
\sup_{\theta \not\in A} \prob_{X|\theta}\bigl\{ \bel_X(A;\S) \geq 1-\alpha \bigr\} \leq \alpha.
\end{equation}
The IM is \emph{valid} if it is valid for all $A$.
\end{defn}

By \eqref{eq:uev}, the validity property can also be stated in terms of the plausibility function.  That is, the IM is valid if, for all assertions $A$ and for any $\alpha \in (0,1)$, 
\begin{equation}
\label{eq:cred.uev}
\sup_{\theta \in A} \prob_{X|\theta}\bigl\{ \pl_X(A;\S) \leq \alpha \bigr\} \leq \alpha.
\end{equation}

\begin{thm}
\label{thm:cred}
Suppose the predictive random set $\S$ is valid, and $\Theta_x(\S) \neq \varnothing$ with $\prob_{\S}$-probability 1 for all $x$.  Then the IM is valid.
\end{thm}

\begin{proof}[\indent Proof]
For any $A$, take $(x,u,\theta)$ such that $\theta \not\in A$ and $x=a(\theta,u)$.  Since $A \subseteq \{\theta\}^c$, $\bel_x(A;\S) \leq \bel_x(\{\theta\}^c;\S) = \prob_\S\{\Theta_x(\S) \not\ni \theta\} = \prob_\S\{\S \not\ni u\}$ by monotonicity.  Validity of $\S$ implies that the right-hand side, as a function of $U \sim \prob_U$, is stochastically smaller than $\unif(0,1)$.  This, in turn, implies the same of $\bel_X(A;\S)$ as a function of $X \sim \prob_{X|\theta}$.  Therefore, $\prob_{X|\theta}\{\bel_X(A;\S) \geq 1-\alpha \} \leq \prob\{ \unif(0,1) \geq 1-\alpha \} = \alpha$.  Taking a supremum over $\theta \not\in A$ on the left-hand side completes the proof. 
\end{proof}

A key feature of the validity theorem above is that it holds under minimal conditions on the predictive random set.  Validity of the IM does not depend on the particular form of predictive random set, only that it is valid.  Recall that the condition ``$\Theta_x(\S) \neq \varnothing$ with $\prob_\S$-probability~1'' holds whenever \eqref{eq:constraint} holds.  See, also, \citet{leafliu2012}.

The following corollary states that the validity theorem remains true even after a suitable---possibly $\theta$-dependent---change of auxiliary variable.  In other words, the validity property is independent of the choice of auxiliary variable parametrization.  This reparametrization comes in handy in examples, including those in Section~\ref{S:more-examples}.  

\begin{cor}
\label{cor:repar}
Consider a one-to-one transformation $v = \phi_\theta(u)$ such that the push-forward measure $\prob_V = \prob_U \phi_\theta^{-1}$ on $\VV = \phi_\theta(\UU)$ does not depend on $\theta$.  Suppose $\S$ is valid for predicting $v^\star = \phi_\theta(u^\star)$, and $\Theta_x(\S) \neq \varnothing$ with $\prob_{\S}$-probability 1 for all $x$.  Then the corresponding belief function satisfies \eqref{eq:cred.lev} and the transformed IM is valid.  
\end{cor}

\subsection{Application: IM-based frequentist procedures}
\label{SS:freq}

In addition to providing problem-specific measures of certainty about various assertions of interest, the belief/plausibility functions can easily be used to create frequentist procedures.  First consider testing $H_0: \theta \in A$ versus $H_1: \theta \in A^c$.  Then an IM-based counterpart to a frequentist testing rule is of the following form: 
\begin{equation}
\label{eq:test}
\text{Reject $H_0$ if $\pl_x(A) \leq \alpha$, for a specified $\alpha \in (0,1)$}. 
\end{equation}
According to \eqref{eq:cred.uev} and Theorem~\ref{thm:cred}, if the predictive random set $\S$ is valid, then the probability of a Type~I error for such a rejection rule is $\sup_{\theta \in A} \prob_{X|\theta}\{\pl_X(A) \leq \alpha\} \leq \alpha$.  Therefore, the test \eqref{eq:test} controls the probability of a Type~I error at level $\alpha$.  

Next consider the class of singleton assertions $\{\theta\}$, with $\theta \in \Theta$.  As a counterpart to a frequentist confidence region, define the $100(1-\alpha)$\% \emph{plausibility region}
\begin{equation}
\label{eq:evid.int}
\Pi_x(\alpha) = \{\theta: \pl_x(\theta) > \alpha \}. 
\end{equation}
Now the coverage probability of the plausibility region \eqref{eq:evid.int} is 
\[ \prob_{X|\theta}\{\Pi_X(\alpha) \ni \theta\} = \prob_{X|\theta}\{\pl_X(\theta) > \alpha\} = 1-\prob_{X|\theta}\{\pl_X(\theta) \leq \alpha\} \geq 1-\alpha, \]
where the last inequality follows from Theorem~\ref{thm:cred}.  Therefore, this plausibility region has at least the nominal coverage probability.  

\begin{gex}[cont]
Suppose $X=5$.  Then, using the predictive random set $\S$ in \eqref{eq:default.prs}, the plausibility function is $\pl_5(\theta) = 1-|2\Phi(5-\theta)-1|$.  The 90\% plausibility interval for $\theta$, determined by the inequality $\pl_5(\theta) > 0.10$, is $5 \pm \Phi^{-1}(0.05)$, the same as the classical 90\% $z$-interval for $\theta$ given in standard textbooks.  
\end{gex}

\begin{pex}[cont]
For the predictive random set determined by $\S$ in \eqref{eq:default.prs}, the plausibility function $\pl_x(\theta)$ is displayed in \eqref{eq:pois-evid}.  For observed $X=5$, a 90\% plausibility interval for $\theta$, characterized by the inequality $\pl_5(\theta) > 0.10$, is $(1.97, 10.51)$. This interval is not the best possible; in fact, the one presented in Section~\ref{SSS:two.sided} is better.  But these plausibility intervals have exact coverage properties, which means that they may be too conservative at certain $\theta$ values for practical use.  This is the case for all exact intervals in discrete data problems \citep[e.g.,][]{bcd2003, cai2005}. 
\end{pex}


\section{Theoretical optimality of IMs}
\label{S:optimality}

\subsection{Intuition}
\label{SS:efficiency.intuition}

\citet{mzl2010} showed that Fisher's fiducial inference and Dempster--Shafer theory are special cases of the IM framework corresponding to a singleton predictive random set.  But it is easy to show that, for some assertions $A$, the fiducial probability is not valid in the sense of Definition~\ref{def:validity}.  To correct for this bias, we propose to replace the singleton with some larger $\S$.  But taking $\S$ to be too large will lead to inefficient inference.  So the goal is to take $\S$ just large enough that validity is achieved.

\subsection{Preliminaries}
\label{SS:prelim}

Throughout the subsequent discussion, we shall assume \eqref{eq:constraint}, i.e., $\Theta_x(u)\neq\varnothing$ for all $x$ and $u$.  This allows us to ignore conditioning in the definition of belief functions.  

For the predictive random set $\S_0 = \{U\}$, with $U \sim \unif(0,1)$, the belief function at $A$ is $\bel_x(A; \S_0) = \prob_U\{\Theta_x(U) \subseteq A\}$, where $\Theta_x(u) = \{\theta: x=a(\theta,u)\}$ as in \eqref{eq:inverse1}.  This is exactly the fiducial probability for $A$ given $X=x$.  For a general predictive random set $\S$, we have $\bel_x(A; \S) = \prob_\S\{\Theta_x(\S) \subseteq A\}$, where $\Theta_x(\S) = \bigcup_{u \in \S} \Theta_x(u)$ is defined in \eqref{eq:inverse2}.  In light of the discussion in Section~\ref{SS:efficiency.intuition}, we shall compare the two belief functions $\bel_x(A; \S)$ and $\bel_x(A; \S_0)$.  Towards this, we have the following result which says that the fiducial probability is an upper bound for the belief function.  

\begin{prop}
\label{prop:bel.fid.bound}
If \eqref{eq:constraint} holds and the predictive random set $\S$ is valid in the sense of Definition~\ref{def:credible}, then $\bel_x(A; \S) \leq \bel_x(A; \S_0)$ for each fixed $x$.  
\end{prop}

\begin{proof}[\indent Proof]
Let $\UU_x(A) = \{u: \Theta_x(u) \subseteq A\}$; note that $\S \subseteq \UU_x(A)$ iff $\Theta_x(\S) \subseteq A$.  Also, put $b = \bel_x(A;\S)$ and $b_0 = \bel_x(A; \S_0) \equiv \prob_U\{\UU_x(A)\}$.  If $u \not\in \UU_x(A)$, then 
\[ Q_\S(u) = \prob_\S\{\S \not\ni u\} \geq \prob_\S\{\S \subseteq \UU_x(A)\} = \prob_\S\{\Theta_x(\S) \subseteq A\} = b. \]
Therefore, $\prob_U\{Q_\S(U) \geq b\} \geq \prob_U\{\UU_x(A)^c\} = 1-b_0$.  Also, validity of $\S$ implies $\prob_U\{Q_\S(U) \geq b\} \leq 1-b$.  Consequently, $1-b_0 \leq 1-b$, i.e., $\bel_x(A; \S) \leq \bel_x(A; \S_0)$.  
\end{proof}

\ifthenelse{1=1}{}{
\begin{proof}[\indent Proof]
The proof is by contradiction.  First write $\bel_x(A; \S) = \prob_\S\{\S \subseteq \UU_x(A)\}$.  Suppose that $\delta = \prob_\S\{\S \subseteq \UU_x(A)\} - \prob_U\{\UU_x(A)\}$ is greater than zero.  Set $\alpha = 1-\prob_\S\{\S \subseteq \UU_x(A)\}$.  Then for any $u \not\in \UU_x(A)$, we have $\prob_\S\{\S \not\ni u\} > \prob_\S\{\S \subseteq \UU_x(A)\} = 1-\alpha$.  If, as in Section~\ref{SS:credible}, we write $Q_\S(u)$ for the left-hand side above, then it follows that 
\begin{align*}
\prob_U\{Q_\S(U) > 1-\alpha\} & \geq \prob_U\{\UU_x(A)^c\} = 1 - \prob_U\{\UU_x(A)\} \\
& = 1 - \bigl[ \prob_\S\{\S \subseteq \UU_x(A)\} - \delta \bigr] \\
& = \alpha + \delta.
\end{align*}
Since $\delta > 0$ by supposition, it follows that $\S$ is not valid---a contradiction!  Therefore, for each $x$, $\bel_x(A; \S)$ cannot exceed $\prob_U\{\UU_x(A)\}$, the fiducial probability for $A$.  
\end{proof}
}

For given assertion $A$ and predictive random set $\S$, consider the ratio
\begin{equation}
\label{eq:R.ratio}
R_A(x; \S) = \bel_x(A; \S) / \bel_x(A; \S_0), \quad x \in \XX. 
\end{equation}
We call this the relative efficiency of the IM based on $\S$ compared to fiducial.  Proposition~\ref{prop:bel.fid.bound} guarantees that this ratio is bounded by unity, provided that the denominator $\bel_x(A; \S_0)$ is non-zero.  Our main goal is to choose $\S$ to make this ratio large in some sense.  Towards this goal, we have the following ``complete-class theorem'' which says that nested predictive random sets---which, by Theorems~\ref{thm:nested.prs} and \ref{thm:cred}, produce valid IMs---are the only kind of predictive random sets under consideration.  

\begin{thm}
\label{thm:complete.class}
Fix $A \subseteq \Theta$ and assume \eqref{eq:constraint}.  Given any predictive random set $\S$, there exists a nested predictive random set $\S'$ such that $R_A(x;\S') \geq R_A(x; \S)$ for each $x$.  
\end{thm}

\begin{proof}[\indent Proof]
Given $\S$, construct a collection $\SS = \{S_x: x \in \XX\}$ as follows:
\[ S_x = \bigcap_{x': \bel_{x'}(A; \S) \geq \bel_x(A; \S)} \UU_{x'}(A),\]
where $\UU_x(A)$ is defined in the proof of Proposition~\ref{prop:bel.fid.bound}.  This collection $\SS$, which will serve as the support for the new $\S'$, is clearly nested.  Indeed, if $\bel_{x_1}(A; \S) \leq \bel_{x_2}(A; \S)$, then $S_{x_2} \subseteq S_{x_1}$.  The distribution $\prob_{\S'}$ of $\S'$ is defined as 
\[ \prob_{\S'}\{\S' \subseteq K\} = \sup_{x: S_x \subseteq K} \bar b(x), \quad K \subseteq \UU, \]
where $\bar b(t) = \bel_t(A;\S) / \sup_x \bel_x(A;\S)$ is the normalized belief function.  In particular, for $K=S_x$, we have $\prob_{\S'}\{\S' \subseteq S_x\} = \bar b(x) \geq \bel_x(A;\S)$.  Then we have
\begin{align*}
\bel_x(A; \S') & = \prob_{\S'}\{\Theta_x(\S') \subseteq A\} \\
& = \prob_{\S'}\{\S' \subseteq \UU_x(A)\} \\
& \geq \prob_{\S'}\{\S' \subseteq S_x\} \\
& \geq \bel_x(A;\S), 
\end{align*}
where the second equality is due to the fact that $\Theta_x(\S') \subseteq A$ iff $\S' \subseteq \UU_x(A)$, and the first inequality is by monotonicity of $\prob_{\S'}\{\S' \subseteq \cdot\}$ and the fact that $S_x \subseteq \UU_x(A)$ for each $x$.  Therefore, $R_A(x; \S')$ can be no less than $R_A(x; \S)$ for each $x$, proving the claim.
\end{proof}

We omit the details, but if the collection $\UU_x(A)$ in the proof of Proposition~\ref{prop:bel.fid.bound} is itself nested, then, in general, one can construct an ``optimal'' $\S$ such that $R_A(X;\S) \equiv 1$.  This is done explicitly for the special case in Section~\ref{SSS:one.sided} below.

\subsection{Optimality in special cases}
\label{SS:special.cases}

Throughout this section, we will focus on scalar $X$ and $\theta$.  However, this is just for simplicity, and not a limitation of the method; see Section~\ref{S:more-examples}.  Indeed, special dimension-reduction techniques, akin to Fisher's theory of sufficient statistics, are available to reduce the dimension of observed $X$ to that of $\theta$ within the IM framework; see \citet{imcond, immarg}.  Also, there is no conceptual difference between scalar and vector $\theta$ problems so, since the ideas are new, we prefer to keep the presentation as simple as possible. 

\subsubsection{One-sided assertions}
\label{SSS:one.sided}

Here we consider a one-sided assertion, e.g., $A = \{\theta \in \Theta: \theta < \theta_0\}$, where $\theta_0$ is fixed.  This ``left-sided'' assertion is the kind we shall focus on, but other one-sided assertions can be handled similarly.  In this context, we can consider a very strong definition of optimality.  

\begin{defn}
\label{def:optimal}
Fix a left-sided assertion $A$.  For two nested predictive random sets $\S$ and $\S'$, the IM based on $\S$ is said to be more efficient than that based on $\S'$ if, as functions of $X \sim \prob_{X|\theta}$ for any $\theta \in A$, $R_A(X;\S)$ is stochastically larger than $R_A(X;\S')$.  The IM based on $\S^\star$ is optimal, or most efficient, if $R_A(X; \S^\star)$ is stochastically largest, in the sense above, among all nested predictive random sets.  
\end{defn}

That optimality here is described via a stochastic ordering property is natural in light of the notion of validity used throughout.  This definition particularly strong because it concerns the full distribution of $R_A(X,\S)$ as a function of $X \sim \prob_{X|\theta}$, not just a functional thereof.  Next we establish a strong optimality result for one-sided assertions; when the assertion is not one-sided, it may not be possible to establish such a strong result.  

\begin{thm}
\label{thm:one.sided}
Let $A = \{\theta \in \Theta: \theta < \theta_0\}$ be a left-sided assertion.  Suppose that $\Theta_x(u)$, defined in \eqref{eq:inverse1}, is such that, for each $x$, the right endpoint $\sup \Theta_x(u)$ is a non-decreasing $($resp.~non-increasing$)$ function of $u$.  Then, for the given $A$, the optimal predictive random set is $\S^\star = [0,U]$ $($resp.~$\S^\star=[U,1]$$)$, where $U \sim \unif(0,1)$.  
\end{thm}

\begin{proof}[\indent Proof]
First observe that both forms of $\S^\star$ are nested.  We shall focus on the non-decreasing case only; the other case is similar.  Since $\sup\Theta_x(u)$ is non-decreasing in $u$, it follows that $\sup \Theta_x([0,U]) = \sup \Theta_x(U)$.  Therefore, 
\[ \bel_x(A; \S^\star) = \prob_U\{\sup \Theta_x([0,U]) < \theta_0\} = \prob_U\{\sup \Theta_x(U) < \theta_0\} = \bel_x(A; \S_0). \]
This holds for all $x$, so $R_A(\cdot; \S^\star) \equiv 1$, its upper bound.  Consequently, $R_A(X; \S^\star)$ is stochastically larger than $R_A(X;\S)$ for any other $\S$, so optimality of $\S^\star$ obtains.  
\end{proof}

\begin{gex}[cont]
We showed previously that $\Theta_x(u) = \{x-\Phi^{-1}(u)\}$.  If we treat this as a degenerate interval, then we see that the right endpoint $x-\Phi^{-1}(u)$ is a strictly decreasing function of $u$.  Therefore, by Theorem~\ref{thm:one.sided}, the optimal predictive random set for a left-sided assertion is $\S^\star=[U,1]$, $U \sim \unif(0,1)$.  

As an application, consider the testing problem $H_0: \theta \geq \theta_0$ versus $H_1: \theta < \theta_0$.  If we take $A = (-\infty,\theta_0)$, then the IM-based rule \eqref{eq:test} rejects $H_0$ iff $1-\bel_x(A; \S^\star) \leq \alpha$.  With the optimal $\S^\star=[U,1]$ as above, we get $\bel_x(A; \S^\star) = \Phi(\theta_0-x)$.  So the IM-based testing rule rejects $H_0$ iff $\Phi(\theta_0-x) \geq 1-\alpha$ or, equivalently, iff $x \leq \theta_0 - \Phi^{-1}(1-\alpha)$.  The reader will recognize this as the uniformly most powerful size-$\alpha$ test based on the classical Neyman--Pearson theory.  
\end{gex}

\begin{pex}[cont]
In this case, $\Theta_x(u) = (G_x^{-1}(u), G_{x+1}^{-1}(u)]$; see \eqref{eq:pois-aeqn-t-mapping}.  The right endpoint $G_{x+1}^{-1}(u)$ is strictly increasing in $u$.  So Theorem~\ref{thm:one.sided} states that, for left-sided assertions, the optimal predictive random set is $\S^\star=[0,U]$, $U \sim \unif(0,1)$.  The same connection with the Neyman--Pearson uniformly most powerful test in the Gaussian example holds here as well, but we omit the details.  
\end{pex}

\subsubsection{Two-sided assertions}
\label{SSS:two.sided}

Consider the case where $A = \{\theta_0\}^c$ is the two-sided assertion of interest, with $\theta_0$ a fixed interior point of $\Theta \subseteq \RR$.  This is an important case, which we have already considered in Section~\ref{S:im}, just in a different form.  These problems are apparently more difficult than their one-sided counterparts, just like in the classical hypothesis testing context.  Here we present some basic results and intuitions on local IM optimality for two-sided assertions.  

Assume $\prob_{X|\theta}$ is continuous.  Then the fiducial probability $\bel_X(\{\theta_0\}^c;\S_0)$ for the two-sided assertion is unity, and so the relative efficiency \eqref{eq:R.ratio} is simply $\bel_x(\{\theta_0\}^c;\S)$.  Here we focus on predictive random sets $\S$ with the property that $\bel_X(\{\theta_0\}^c;\S) \sim \unif(0,1)$ under $\prob_{X|\theta_0}$; see Corollary~\ref{thm:prs}.   Based on the intuition developed in Section~\ref{S:im}, $\bel_X(\{\theta_0\}^c;\S)$ should be smallest (probabilistically) under $\prob_{X|\theta}$ for $\theta=\theta_0$.  We shall, therefore, impose the following condition on the predictive random set $\S$:
\begin{equation}
\label{eq:ordering}
\prob_{X|\theta} \{ \bel_X(\{\theta_0\}^c; \S) \leq \alpha\} < \alpha, \quad \forall\;\theta\neq\theta_0, \quad \forall \; \alpha \in (0,1).  
\end{equation}
Roughly speaking, condition \eqref{eq:ordering} states that the belief function at $\{\theta_0\}^c$ is stochastically larger under $\prob_{X|\theta}$ than under $\prob_{X|\theta_0}$.  There is also a loose connection between \eqref{eq:ordering} and the classical unbiasedness condition imposed to construct optimal tests when the alternative hypothesis is two-sided \citep[][Ch.~4]{lehmann.romano.2005}.  Our goal in what follows is to find a ``best'' predictive random set that satisfies \eqref{eq:ordering}.

To make things formal, suppose that both $\XX$ and $\Theta$ are one-dimensional, that $\prob_{X|\theta}$ is continuous with distribution function $F_\theta(x)$ and density function $f_\theta(x)$, and that the usual regularity conditions hold; in particular, we assume that the order of expectation with respect to $\prob_{X|\theta}$ and differentiation with respect to $\theta$ can be interchanged.  Note that we have fixed a parametrization, and the analysis that follows depends on this selection.  Let $T_\theta(x) = (\del/\del\theta) \log f_\theta(x)$ be the score function, an important quantity in what follows.  Also, let $V_\theta(x) = T_\theta(x)^2 + (\del/\del\theta)T_\theta(x)$.  Then, under the usual regularity conditions, we have $\E_{X|\theta}\{T_\theta(X)\} = 0$ and $\E_{X|\theta}\{V_\theta(X)\} = 0$ for all $\theta$.    

In Appendix~\ref{S:details} we argue that a good predictive random set $\S$ must have a support with certain symmetry or balance properties with respect to the sampling distribution of $T_{\theta_0}(X)$.  In particular, let $B=\{B_t: t \in \TT\}$ be a generic collection of nested measurable subsets of $\TT = T_{\theta_0}(\XX)$.  The collection $B$ shall be called \emph{score-balanced} if 
\begin{equation}
\label{eq:balance1a}
\E_{X|\theta_0}\{T_{\theta_0}(X) I_{B_t}(T_{\theta_0}(X))\} = 0, \quad \forall\; t \in \TT. 
\end{equation} 

For a score-balanced collection $B=\{B_t\}$ satisfying \eqref{eq:balance1a} we can define a corresponding \emph{score-balanced predictive random set} $\S=\S_B$ as follows.  Define the class $\SS = \{S_t: t \in \TT \}$ of subsets of $\UU = [0,1]$ given by 
\[ S_t = F_{\theta_0}\bigl(\{x: T_{\theta_0}(x) \in B_t\}). \]
For simplicity, and without loss of generality, assume $\SS$ contains $\varnothing$ and $\UU$.  Now take a predictive random set $\S_B$, supported on $\SS$, such that its measure $\prob_{\S_B}$ satisfies 
\[ \prob_{\S_B}\{\S_B \subseteq K\} = \sup_{t: S_t \subseteq K} \prob_U(S_t), \quad K \subseteq [0,1], \]
where $\prob_U$ is the $\unif(0,1)$ measure.  (The set $S_t$ is $\prob_U$-measurable for all $t$ by the assumed measurability of $B_t$, $T_{\theta_0}$, and $F_{\theta_0}$.)  The corresponding score-balanced belief function is 
\begin{align*}
\bel_x(\{\theta_0\}^c; \S_B) & = \prob_{\S_B}\{\S_B \not\ni F_{\theta_0}(x)\} \\
& = \prob_{X|\theta_0}\{B_{T_{\theta_0}(X)} \not\ni T_{\theta_0}(x)\} \\
& = \prob_{X|\theta_0}\{T_{\theta_0}(X) \in B_{T_{\theta_0}(x)}\}, 
\end{align*}
where the last equality follows from the assumed nesting of $\{B_t\}$.  Proposition~\ref{prop:balance1} in Appendix~\ref{S:details} shows that predictive random sets which are good in the sense that \eqref{eq:ordering} holds (at least locally) must be score-balanced. 

But there are many such $\S_B$ to choose from, so we now consider finding a ``best'' one.  A reasonable definition of optimal score-balanced predictive random set is one that makes the difference between the right- and left-hand sides of \eqref{eq:ordering} as large as possible for each $\theta$ in a neighborhood of $\theta_0$.  Then, for two-sided assertions, we have     

\begin{defn}
\label{def:two.sided}
Let $B^\star=\{B_t^\star: t \in \TT\}$ be such that, for each $t$, 
\begin{equation}
\label{eq:ts.optimal}
\int_{T_{\theta_0}(x) \in B_t^\star} V_{\theta_0}(x) f_{\theta_0}(x) \,dx 
\end{equation}
is minimized subject to the score-balance constraint \eqref{eq:balance1a}.  Then $\S^\star=\S_{B^\star}$ is the optimal score-balanced predictive random set.  
\end{defn}

Here we give a general construction of an an optimal score-balanced predictive random sets.  Proving that the predictive random sets satisfy the conditions of Definition~\ref{def:two.sided} will require assumptions about the model.  Start with the following class of intervals:
\begin{equation}
\label{eq:B.interval}
B_t^\star = \bigl(\xi_-(t), \xi_+(t)\bigr), \quad t \in T_{\theta_0}(\XX), 
\end{equation}
where the functions $\xi_-,\,\xi_+$ (which depend implicitly on $\theta_0$) are such that \eqref{eq:balance1a} holds.  In addition, we shall assume these functions are continuous and satisfy 
\begin{itemize}
\item $\xi_-(t)$ is non-positive, $\xi_-(t) = t$ for $t \in (-\infty,0)$ and is decreasing for $t \in [0,\infty)$;
\vspace{-2mm}
\item $\xi_+(t)$ is non-negative, $\xi_+(t) = t$ for $t \in [0,\infty)$ and is increasing for $t \in (-\infty, 0)$.
\end{itemize}
The functions $\xi_-,\,\xi_+$ describe a sort of symmetry/balance in the distribution of $T_{\theta_0}(X)$: they satisfy $\xi_+(\xi_-(t)) = t$ and $\xi_-(\xi_+(-t)) = -t$ for all $t \geq 0$.  In some cases, for given $t$, expressions for $\xi_-(t)$ and $\xi_+(t)$ can be found analytically, but typically numerical solutions are required.  Set $\S^\star = \S_{B^\star}$.  We claim that, under certain conditions on $V_{\theta_0}(x)$, $\S^\star$ is optimal in the sense of Definition~\ref{def:two.sided}.  

Before we get to the optimality considerations, we first verify the assumption that $\bel_X(\{\theta_0\}^c; \S^\star) \sim \unif(0,1)$ under $\prob_{X|\theta_0}$.  From the definition of $B_t^\star$, it is clear that 
\begin{align*}
T \in B_t^\star & \iff \xi_-(t) < T < \xi_+(t) \\
& \iff \xi_-(t) < \xi_-(T) < \xi_+(T) < \xi_+(t) \\
& \iff \xi_+(T) - \xi_-(T) < \xi_+(t) - \xi_-(t).
\end{align*}
Consequently, if $D_{\theta_0}(X) = \xi_+(T_{\theta_0}(X)) - \xi_-(T_{\theta_0}(X))$, then 
\[ \bel_x(\{\theta_0\}^c; \S^\star) = \prob_{X|\theta_0}\{T_{\theta_0}(X) \in B_{T_{\theta_0}}^\star(x)\} = \prob_{X|\theta_0}\{D_{\theta_0}(X) < D_{\theta_0}(x)\}. \]
Therefore, since $D_{\theta_0}(X)$ is a continuous random variable, an argument like that in Corollary~\ref{thm:prs} shows that $\bel_X(\{\theta_0\}^c;\S^\star) \sim \unif(0,1)$ under $\prob_{X|\theta_0}$.  

We are now ready for optimality of $\S^\star$.  Write $V(t)$ for $V_{\theta_0}(x)$, when treated as a function of $t=T_{\theta_0}(x)$.  The condition to be imposed is:
\begin{equation}
\label{eq:unimodal}
\text{$V(t)$ is uniquely minimized at $t=0$, and $V(0) < 0$.} 
\end{equation}
This condition holds, e.g., for all exponential families with $\theta$ the natural parameter.  

\begin{prop}
\label{prop:two.sided}
Under condition \eqref{eq:unimodal}, the score balanced predictive random set $\S^\star=\S_{B^\star}$, with $B^\star$ described above, is optimal in the sense of Definition~\ref{def:two.sided}.  
\end{prop}

\begin{proof}[\indent Proof]
The proof is simple but tedious so here we just sketch the main idea.  Under \eqref{eq:unimodal}, the intervals $B_t^\star$ which are ``balanced'' around $T_{\theta_0}(x) = 0$, make most efficient use of the space where $V_{\theta_0}(x)$ is smallest in the following sense.  They are exactly the right size to make $\S_{B^\star}$ efficient, so any other efficient score-balanced predictive random set $\S_B$ must be determined by sets $B=\{B_t\}$ other than intervals concentrated around $T_{\theta_0}(x)=0$.  Since such intervals are where $V_{\theta_0}(x)$ is smallest, the integral in \eqref{eq:ts.optimal} corresponding to $B_t$ must be larger than that corresponding to $B_t^\star$.  Therefore, $\S^\star$ satisfies the conditions of Definition~\ref{def:two.sided} and, hence, is optimal. \end{proof} 

Unfortunately, \eqref{eq:unimodal} is not always satisfied.  For example, it can fail for exponential families not in natural form.  But we claim that \eqref{eq:unimodal} is not absolutely essential.  Assume $V(t)$ is convex and $V(0) < 0$.  This relaxed assumption holds, e.g., for all exponential families.  To keep things simple, suppose that $V(t)$ is minimized at $\hat t > 0$.  Although the argument to be given is general, Figure~\ref{fig:exp}(a) illustrates the phenomenon for the exponential distribution with mean $\theta_0=1$.  The heavy line there represents $V(t)$, and the thin lines represent $th(t)$ (black) and $V(t)h(t)$ (gray), where $h(t)$ is the density of $T$.  The horizontal lines represent the intervals $B_t^\star$ in \eqref{eq:B.interval} for select $t$.  By convexity of $V(t)$, there exists $t_0$ such that $\hat t \in (0,t_0)$ and $V(t) < V(0)$ for each $t \in (0,t_0)$; this is $(0,0.5)$ in the figure.  For $t \in (0,t_0)$, the intervals $B_t^\star$ do not contain $(0,t_0)$; these intervals are shown in black.  In such cases, the integral \eqref{eq:ts.optimal} can be reduced by breaking $B_t^\star$ into two parts: one part takes more of $(0,t_0)$, where $V(t)$ is smallest, and the other part is chosen to satisfy the score-balance condition \eqref{eq:balance1a}.  But when $t \geq t_0$, no improvement can be made by changing $B_t^\star$; these cases are shown in gray.  So, in this sense, the intervals $B_t^\star$ in \eqref{eq:B.interval} are not too bad even if \eqref{eq:unimodal} fails.  

On the other hand, violations of \eqref{eq:unimodal} are due to the choice of the parametrization.  Indeed, under mild assumptions, there exists a transformation $\eta=\eta(\theta)$ such that the corresponding $V(t)$ function for $\eta$ satisfies \eqref{eq:unimodal}. Then the predictive random set $\S^\star$ in Proposition~\ref{prop:two.sided} is the optimal for this transformed problem.

\begin{gex}[cont]
This is a natural exponential family distribution, so Proposition~\ref{prop:two.sided} holds, and $\S^\star$ is the optimal score-balanced predictive random set.  Here the score function is $T_\theta(x) = x-\theta$.  Under $X \sim \nm(\theta,1)$, the distribution of $T_\theta(X)$ is symmetric about 0.  Therefore, $B_t^\star = (-|t|,|t|)$, and the corresponding predictive random set is supported on subsets $S_t$ given by 
\[ S_t = F_{\theta_0}\bigl(\{x: |x-\theta_0| \leq |t|\}\bigr) = \bigl( \Phi(-|t|), \Phi(|t|) \bigr), \]
with belief function $\bel_x(\{\theta_0\}^c; \S^\star) = 2\Phi(|x-\theta_0|)-1$.  This is exactly one minus the plausibility function in \eqref{eq:gauss.pl} based on the default predictive random set \eqref{eq:default.prs}.  Therefore, we conclude that the \eqref{eq:default.prs} is, in fact, the optimal score-balanced predictive random set in the Gaussian problem.  This is consistent with our intuition, given that the results based on this default choice in the Gaussian example match up with good classical results.  
\end{gex}

\begin{eex}
Suppose $X$ is an exponential random variable with mean $\theta$, as discussed above.  Unlike the Gaussian, this distribution is asymmetric, so, for the optimal score-balanced IM, a numerical method is needed to identify the set $B_{T_{\theta_0}(x)}$ for each observed $x$.  Plots of the corresponding plausibility functions $\pl_x(\theta;\S) = 1-\bel_x(\{\theta\}^c; \S)$ for two different predictive random sets based on $X=5$ are shown in Figure~\ref{fig:exp}(b).  The black line is based on the optimal score-balanced predictive random set, and the gray line is based on the default predictive random set in \eqref{eq:default.prs}.  90\% plausibility intervals, determined by the horizontal line at $\alpha=0.1$, are much shorter for the score-balanced IM compared to the default in this case.  For comparison, one might consider a crude nominal 90\% confidence interval for $\theta$, namely, $(Xe^{-1.65}, Xe^{1.65})$, based on a variance-stabilizing transformation and normal approximation.  These intervals tend to be shorter than both plausibility intervals, but their coverage probability ($\approx 0.82$) is too small.  
\end{eex}

\begin{figure}
\begin{center}
\subfigure[Plot of $V(t)$ vs.~$t$.]{\scalebox{0.6}{\includegraphics{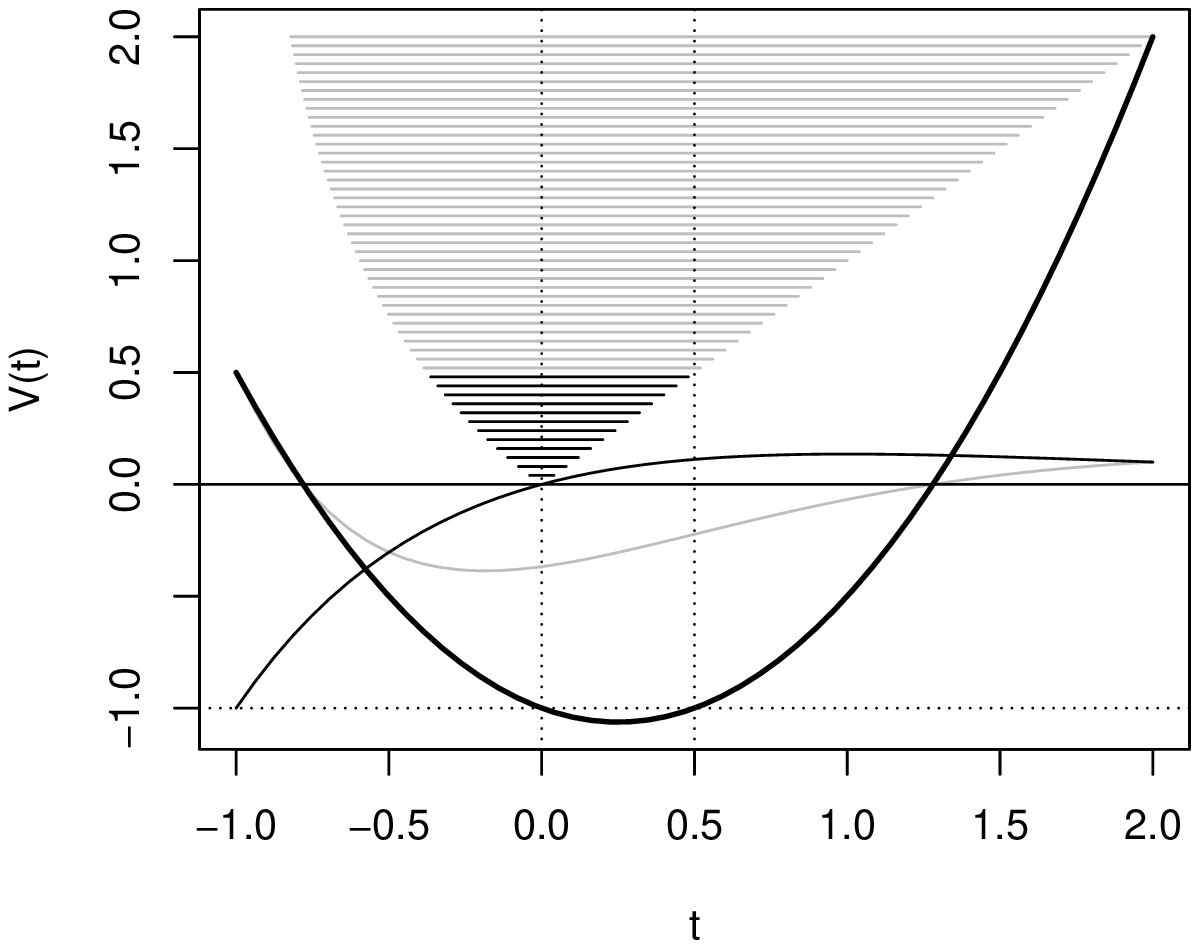}}}
\subfigure[Plausibility functions]{\scalebox{0.6}{\includegraphics{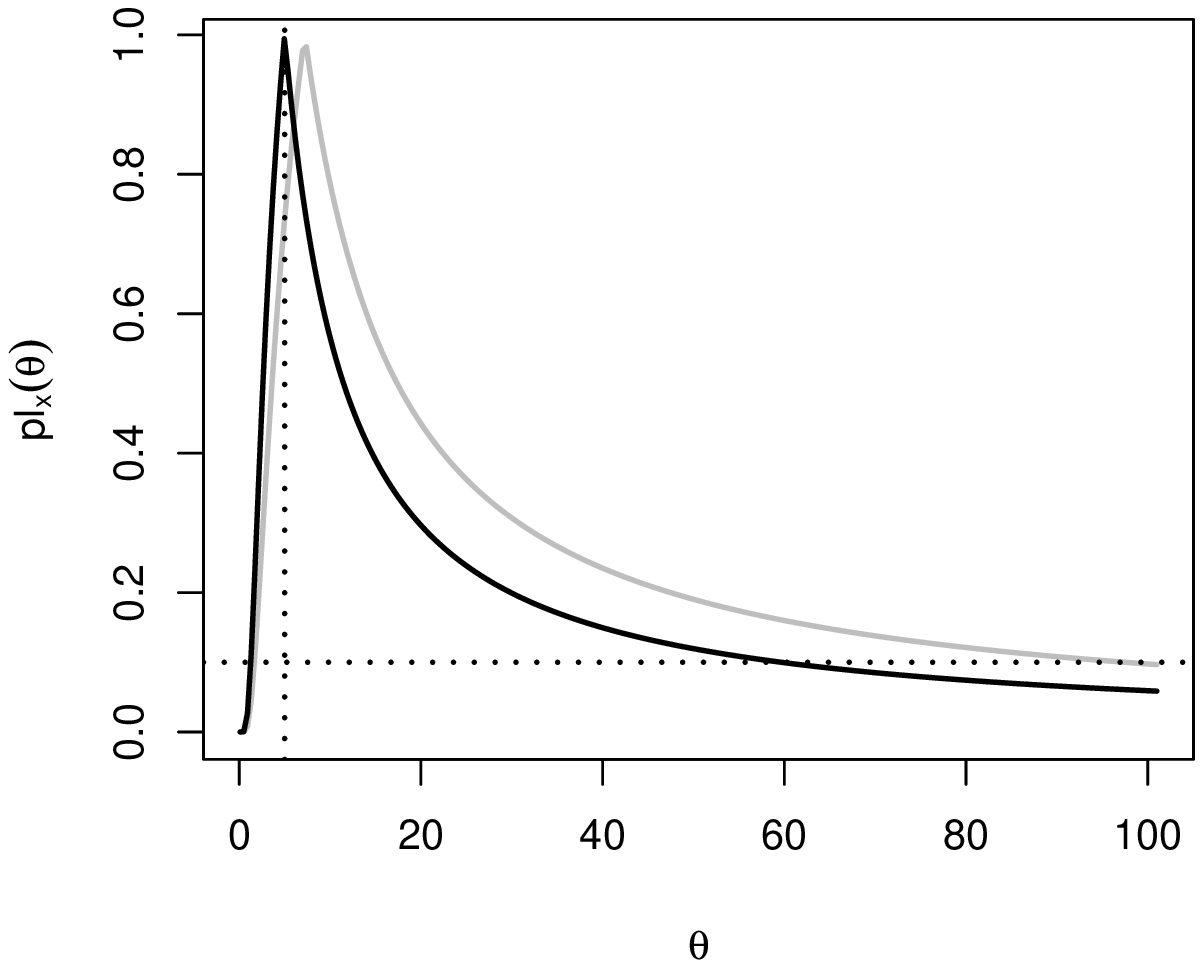}}}
\caption{Specifics of Panel (a) are discussed in the text.  Panel (b) shows $\pl_x(\theta; \S)$, as a function of the exponential scale parameter $\theta$, for two predictive random sets $\S$: optimal score-balanced (black) and default (gray).  Vertical line marks the observed $X=5$.}
\label{fig:exp}
\end{center}
\end{figure} 

\begin{pex}[cont]
Although the theory above holds only for continuous models, the score-balanced predictive random set performs well in discrete problems too.  For the sake of space, we refer the reader to \citet{impois} for the details.  
\end{pex}

\ifthenelse{1=1}{}{
here we focus on the results, saving the computational details to be reported elsewhere \citep{impois}.  Suffice it to say that, evaluation of the plausibility function of the optimal score-balanced IM requires a special iterative sorting of the positive integers, much like what is needed to compute the Poisson mode.  Here the ``optimal'' adjective simply indicates that we are using $\S^\star$ .  Figure~\ref{fig:opt.pois} shows plots of the two plausibility functions $\pl_x(\theta;\S)$---based on score-balanced and default predictive random sets, respectively---for several $x$ values over a range of $\theta$.  Although the shapes of the curves can vary depending on $x$, one general observation is that the score-balanced plausibility function is smaller than the default plausibility function for most $\theta$.  This implies that plausibility intervals based on the former will generally be narrower than those based on the latter.  
\end{pex}

\begin{figure}[t]
\begin{center}
\subfigure[$x=1$]{\scalebox{0.6}{\includegraphics{pois_pl_opt_x1}}}
\subfigure[$x=2$]{\scalebox{0.6}{\includegraphics{pois_pl_opt_x2}}}
\subfigure[$x=5$]{\scalebox{0.6}{\includegraphics{pois_pl_opt_x5}}}
\subfigure[$x=10$]{\scalebox{0.6}{\includegraphics{pois_pl_opt_x10}}}
\caption{Plots of $\pl_x(\theta; \S)$ in the Poisson example for two different $\S$ and several values of $x$: optimal score-balanced optimal (black), and default (gray).}
\label{fig:opt.pois}
\end{center}
\end{figure}
}

\section{Two more examples}
\label{S:more-examples}

\subsection{A standardized mean problem}
\label{SS:normal-example}

Suppose that $X_1,\ldots,X_n$ are independent $\nm(\mu,\sigma^2)$ observations.  The goal is to make inference on $\psi = \mu / \sigma$, the standardized mean, or signal-to-noise ratio.  Following \citet{dempster1963}, we start with a reduction of the full data to the sufficient statistics for $\theta = (\mu,\sigma^2)$, namely $(\Xbar, S^2)$, the sample mean and variance.  Formal IM-based justification for this reduction is available, though we shall not discuss this here.  

For the A-step, we take the association to be
\begin{equation}
\label{eq:normal-a-model}
\Xbar = \mu + n^{-1/2} \sigma U_1 \quad \text{and} \quad S = \sigma U_2, 
\end{equation}
where $U=(U_1,U_2) \sim \prob_U = \nm(0,1) \times \{{\sf ChiSq}(n-1)/(n-1)\}^{1/2}$.  After replacing $\sigma$ in the left-most identity in \eqref{eq:normal-a-model} with $S / U_2$, a bit of algebra reveals that 
\[ n^{1/2} \Xbar / S = (n^{1/2} \psi + U_1) / U_2 \quad \text{and} \quad S = \sigma U_2. \]
For $\theta = (\psi,\sigma)$, make a change of auxiliary variable $v = \phi_\theta(u)$, given by 
\[ v_1 = F_\psi\Bigl( \frac{n^{1/2}\psi + u_1}{u_2} \Bigr) \quad \text{and} \quad v_2 = \frac{\exp(u_2)}{1 + \exp(u_2)}, \]
where $F_\psi$ is the distribution function for $\stt_{n-1}(n^{1/2}\psi)$, a non-central Student-t distribution with $n-1$ degrees of freedom and non-centrality parameter $n^{1/2}\psi$.  Note that the full generality of the parameter-dependent change-of-variables in Corollary~\ref{cor:repar} is needed here.  Then the transformed association is 
\[ n^{1/2} \Xbar / S = F_\psi^{-1}(V_1) \quad \text{and} \quad S = \sigma \log\{V_2 / (1-V_2)\}, \]
and the measure $\prob_V$ on the space of $V = (V_1,V_2)$ has a $\unif(0,1)$ marginal on the $V_1$-space; the distribution on $V_1$-slices of the $V_2$ space can be worked out, but it is not needed in what follows.  For the P-step, we predict $v^\star = \phi_\theta(u^\star)$ with a rectangle predictive random set $\S$ defined by the following set-valued mapping, similar to \eqref{eq:default.prs}:
\begin{equation}
\label{eq:normal-prs}
v = (v_1,v_2) \mapsto \bigl\{v_1': |v_1' - 0.5| < |v_1-0.5| \bigr\} \times [0, 1].
\end{equation}
Optimality considerations along the lines in Section~\ref{SSS:two.sided} could be pursued here, but we choose to keep things simple since analysis of the non-central Student-t distribution is non-trivial.  An important direction of future research is to develop numerical methods for evaluating optimal IMs.  Using a predictive random set that spans the entire $v_2$-space for each $v$ has the effect of ``integrating out'' the nuisance parameter $\sigma$.  For the predictive random set $\S$ in \eqref{eq:normal-prs}, if $z = n^{1/2}\xbar/s$, then the C-step gives the following set $\Theta_x(\S) = \Psi_x(\S) \times \Sigma_x(\S) $ of candidate $(\psi,\sigma)$ pairs:
\begin{equation}
\label{eq:normal-prs-t}
\bigl\{\psi: |F_\psi(z) - 0.5| < |V_1 - 0.5| \bigr\} \times \bigl\{\sigma: \sigma > 0 \bigr\}, \quad V \sim \prob_V.  
\end{equation}
For assertions $A = \{(\psi,\sigma): \sigma > 0\}$ the plausibility function is given by
\[ \pl_x(A) = \prob_{\S}\{\Theta_x(\S) \not\subseteq A^c \} = \prob_{\S}\{\Psi_x(\S) \ni \psi\} = 1 - |2F_\psi(z) - 1|. \]
In this case, the $100(1-\alpha)$\% plausibility interval $\Pi_x(\alpha)$ for $\psi$ is obtained by inverting the inequality $1-|2F_\psi(z)-1| > \alpha$, i.e., $\Pi_x(\alpha) = \{\psi: \alpha/2 < F_\psi(z) < 1-\alpha/2\}$. 

This is exactly the usual frequentist confidence interval based on the sampling distribution of the standardized sample mean; it also agrees with the fiducial intervals obtained by \citet{dempster1963} and \citet{dawidstone1982}.  The standard frequentist approach relies on an informal ``plug-in style'' marginalization, whereas the IM approach above shows exactly how $\sigma$ is ignored via cylinder assertions.  More sophisticated IM marginalization techniques are available, but we do not discuss these here.

\subsection{A many-exponential-rates problem}
\label{SS:poisson.process}

For our last example, we consider a high-dimensional problem.   Suppose that $X=(X_1,\ldots,X_n)$ consists of independent observations $X_i \sim \expo(\theta_i)$, $i=1,\ldots,n$, with unknown rates $\theta_1,\ldots,\theta_n$.  The goal is to give a probabilistic measure of the support in $X=x$ for the assertion $A = \{\theta_1 = \cdots = \theta_n\}$ that the rates are equal.  A version of this problem was also discussed in \citet{mzl2010}, but here we simplify the presentation, emphasize the three-step IM construction, and produce much better results. 

Start, in the A-step, with the association $X_i = U_i / \theta_i$, $i=1,\ldots,n$, where $\prob_U$ is the product measure $\expo(1)^{\times n}$.  Make a change of auxiliary variables $v = \phi(u)$:
\[ v_0 = \textstyle \sum_{i=1}^n u_i \quad \text{and} \quad v_i = u_i / v_0, \quad i=1,\ldots,n. \]
The new vector $v = (v_0,v_1,\ldots,v_n)$ takes values in $\VV = (0,\infty) \times \mathbb{P}_{n-1}$, where $\mathbb{P}_{n-1}$ is the $(n-1)$-dimensional probability simplex in $\RR^n$, and $\prob_V = \prob_U \phi^{-1}$ is the product measure $\gam(n,1) \times \dir_n(1_n)$.  Then the modified association is 
\begin{equation}
\label{eq:pp.amodel}
X_i = V_0 V_i / \theta_i, \quad i=1,\ldots,n, \quad \text{where} \quad V=(V_0,V_1,\ldots,V_n) \sim \prob_V.
\end{equation}
For the P-step, we shall consider the following predictive random set $\S$ characterized by $V \sim \prob_V$ and the set-valued mapping $v \mapsto \{v': h(v') < h(v)\}$.  In this case, we take 
\[ h(v) = -\sum_{i=1}^{n-1} \bigl[a_i \log t_i(v) + b_i \log\{1-t_i(v)\} \bigr], \]
with $t_i(v) = \sum_{j=1}^i v_i$, $a_i = 1/(n-i-0.3)$, and $b_i = 1/(i-0.3)$.  A few remarks on this choice of $\S$ are in order.  First, it follows from Corollary~\ref{thm:prs} that $\S$ is efficient.  Second, the random vector $( t_1(V), \ldots, t_{n-1}(V) )$, for $V \sim \prob_V$, has the distribution of a vector of $n-1$ sorted $\unif(0,1)$ random variables, and \citet[][Sec.~3.4.2]{jzhang-thesis} shows that $\S$ provides an easy-to-compute alternative to the well-performing hierarchical predictive random set for predicting sorted uniforms used in \citet{mzl2010}.  Finally, that the first component $v_0$ of $v$ is essentially ignored in $\S$ is partly for convenience, and partly because $v_0$ is related to the overall scale of the problem which is irrelevant to the assertion $A$ of interest.  

For the C-step, combining the observed data, the association model \eqref{eq:pp.amodel}, and the predictive random set $\S$ above, we get the following random set for $\theta$:
\[ \Theta_x(\S) = \{\theta: h(v(x,\theta)) < h(V) \}, \quad V \sim \prob_V, \]
where $v(x,\theta) = (\theta_1 x_1,\ldots,\theta_nx_n) / \sum_{j=1}^n \theta_j x_j$.  Since the assertion $A = \{\theta_1 = \cdots = \theta_n\}$ is a one-dimensional subset of $\Theta$, the belief function is zero.  It is also important to note that when $\theta$ is a constant vector, $v(x,\theta)$ is independent of that constant, i.e., $v(x,\theta) = v(x,1_n)$, which greatly simplifies computation of the plausibility function at $A$.  Indeed, 
\[ \pl_x(A) = \prob_V\{h(V) > h(v(x,1)) \}, \]
which can easily be evaluated using Monte Carlo.  As described in Section~\ref{SS:freq}, the level $\alpha$ IM-based tests rejects the assertion $A$ if and only if $\pl_x(A) \leq \alpha$.  

For illustration, we compare our results with those of \citet{mzl2010}.  They consider the basic likelihood ratio test, which is based on the test statistic $\bigl\{ \bigl(\textstyle \prod_{i=1}^n x_i\bigr)^{1/n} / \xbar \bigr\}^n$.  They also consider a different sort of IM solution, based on thresholding the plausibility function, but with a default type of predictive random set that uses a Kullback--Leibler neighborhood for predicting the component $(V_1,\ldots,V_n)$ of $V$.  We compare the power of these three tests in several different cases.  In each setup, $n = n_1+n_2=100$ observations are available, but the first $n_1$ exponential rates equal 1 while the last $n_2$ equal $\theta$.  Figure~\ref{fig:exp.sims} shows the power functions over a range of $\theta$ values for two configurations of $(n_1,n_2)$.  Here we see that, in both cases, the likelihood ratio and old IM tests have similar power, possibly because of the common connection to the Kullback--Leibler divergence.  On the other hand, the new IM-based test presented above has strikingly larger power than the other two.  This substantial improvement in power is likely due to the close relationship between our choice of $\S$ and the assertion of interest.  So while the comparison between the new IM results and those of the other ``default'' methods is not entirely fair, it is interesting to see that an assertion-specific choice of predictive random set can lead to drastically improved performance.  

\begin{figure}
\begin{center}
\subfigure[$(n_1,n_2) = (50, 50)$]{\scalebox{0.6}{\includegraphics{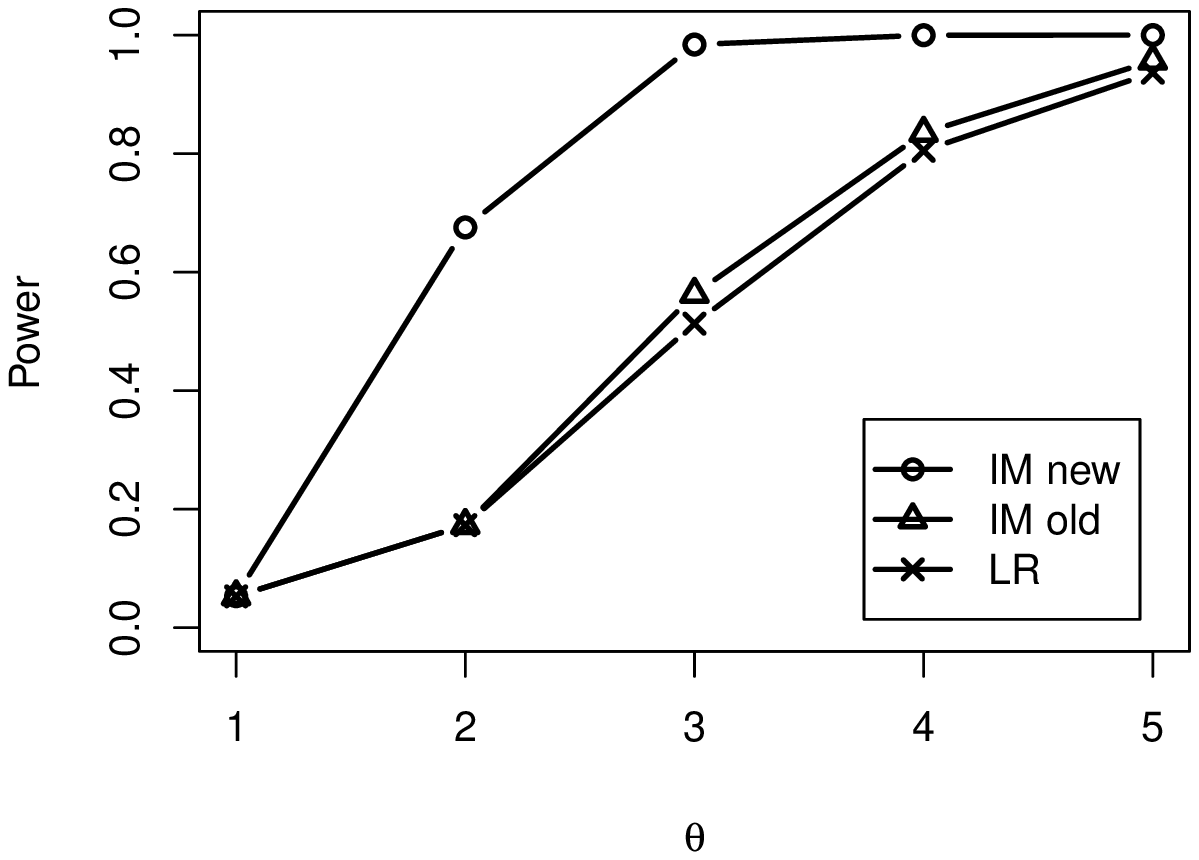}}}
\subfigure[$(n_1,n_2) = (10, 90)$]{\scalebox{0.6}{\includegraphics{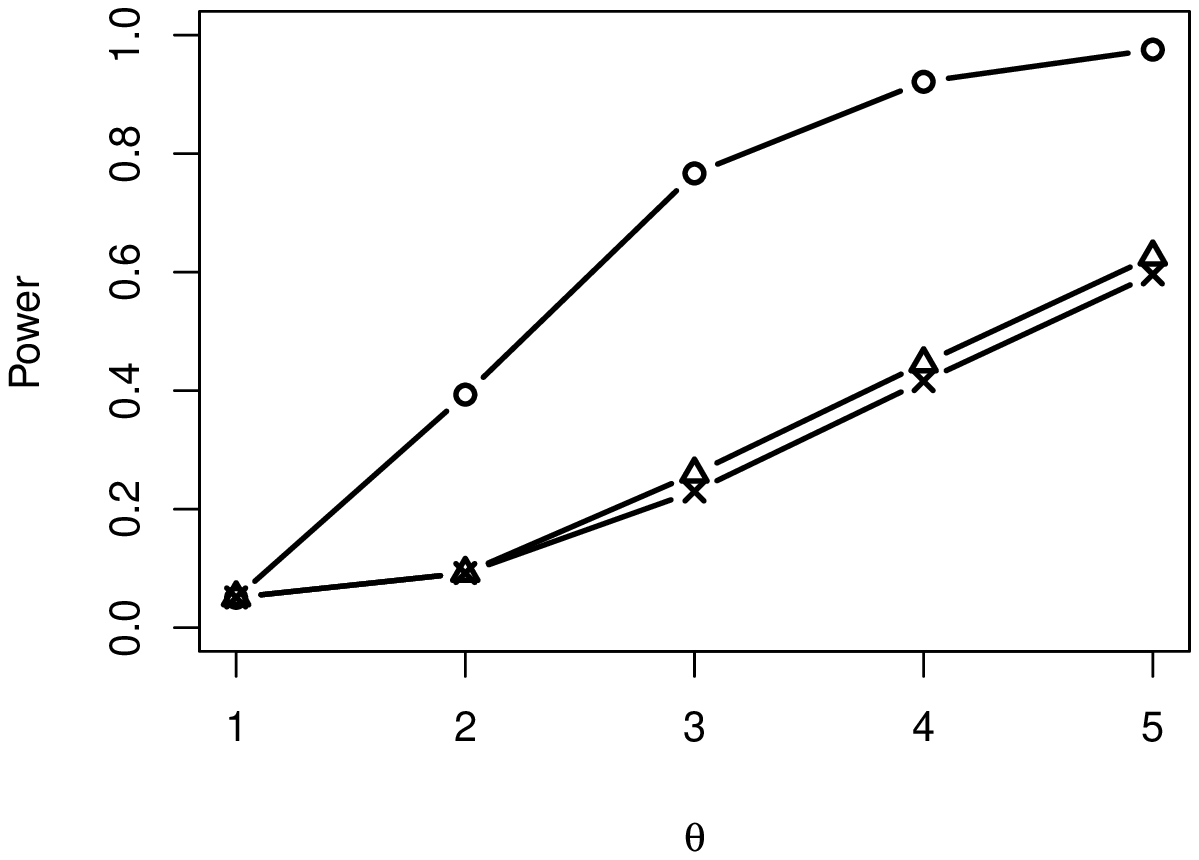}}}
\caption{Estimated powers of the likelihood ratio and two IM-based tests for the simulation described in Section~\ref{SS:poisson.process}. Here $\theta$ is the ratio of the rate of the last $n_2$ observations to that of the first $n_1$. }
\label{fig:exp.sims}
\end{center}
\end{figure}

\section{Discussion}
\label{S:discuss}

The conversion of experience to knowledge is fundamental to the advancement of science, and statistical inference plays a crucial role.  For ages, there has been disagreement about which statistical paradigm to choose.  Both the frequentist and Bayesian paradigms have their own set of advantages and disadvantages, so it would be worthwhile to identify something new which combines the respective advantages but loses, or at least weakens, the disadvantages.  Here we have described a three-step procedure to construct IMs for prior-free, post-data probabilistic inference, and proved that IMs yield frequency-calibrated probabilities under very general conditions.  The point is that the values of the corresponding belief/plausibility function are meaningful both within and across experiments, accomplishing both the frequentist and Bayesian goals simultaneously.  

The proposed IM approach is surely new, but since new is not always better, it is natural to ask what is the benefit of using IMs.  Our response is that, although it will take time for users to familiarize themselves with the thought process, the IM framework is logical, intuitive, and able to produce meaningful and frequency-calibrated probabilistic measures of uncertainty about $\theta$ without a prior distribution.  The latter property is something that no other inferential framework is able to achieve.  

Admittedly, the final IM depends on the user's choice of association and predictive random set, but we do not believe that this is particularly damning.  Section~\ref{S:optimality} laid the foundation for a theory of optimal predictive random sets, and further efforts to develop ``default'' predictive random sets are ongoing, particularly for multi-parameter problems.  But a case can be made to prefer the ambiguity of the choice of predictive random set over that of a frequentist's choice of statistic or Bayesian's choice of prior.  The point is that neither a frequentist sampling distribution nor a Bayesian prior distribution adequately describes the source of uncertainty about $\theta$.  As we argued above, this uncertainty is fully characterized by the fact that, whatever the association, the value of $u^\star$ is missing.  Therefore, it seems only natural to prefer the IM framework that features a direct attack on the source of uncertainty over another that attacks the problem indirectly.  Moreover, as was demonstrated in Section~\ref{SS:poisson.process}, choosing the predictive random set that depends on the problem and/or assertion of interest can lead to drastically improved results.  

We note that differences between IM outputs from different predictive random sets are slight for assertions involving one-dimensional quantities.  However, for high-dimensional auxiliary variables, the choice of predictive random set deserves special attention.  In such cases, our approach is to construct predictive random sets for functions of auxiliary variables that are most relevant to the assertions of interest.  This leads to a practically useful auxiliary variable dimension reduction.  It is interesting that this approach has some close connections to Fisher's theory of sufficient statistics \citep{imcond}.  For nuisance parameter problems, like those in Section~\ref{S:more-examples}, there is a different form of dimension reduction required \citep{immarg}.  

Of course, compared to the well-developed Bayesian and frequentist methods, IMs have many open problems. Both theoretical work and applications have shown that the IM framework is promising.  Given the attractive properties of IMs developed here and in the references above, we expect to see more exciting advancements in IMs or new inferential frameworks \citep[e.g.,][]{plausfn} that are probabilistic and have desirable frequency properties.

\section*{Acknowledgments}

The authors are grateful for the efforts and thoughtful criticisms of the Editor, Associate Editor, and two referees.  This work is partially supported by the U.S. National Science Foundation, grants DMS-1007678, DMS-1208841, and DMS-1208833.

\appendix

\section{Details from Section~\ref{SSS:two.sided}}
\label{S:details}

If we assume that $\bel_X(\{\theta_0\}^c;\S) \sim \unif(0,1)$ under $\prob_{X|\theta_0}$, then there exists a collection of measurable subsets $\XX(\alpha) \subseteq \XX$, depending implicitly on $\theta_0$ and $\S$, such that, for each $\alpha$, $\prob_{X|\theta_0}\{\XX(\alpha)\} = \alpha$, and $\bel_x(\{\theta_0\}^c;\S) \leq \alpha$ iff $x \in \XX(\alpha)$.  It follows that, for any $\theta$, 
\[ \prob_{X|\theta}\{\bel_X(\{\theta_0\}^c; \S) \leq \alpha\} = \psi_\alpha(\theta) :=\int_{\XX(\alpha)} f_\theta(x) \,dx. \]
By definition, $\psi_\alpha(\theta_0) = \alpha$.  Now, \eqref{eq:ordering} is equivalent to $\psi_\alpha(\theta) < \psi_\alpha(\theta_0)$ for all $\alpha$, or, to put it another way, $\psi_\alpha(\theta)$ is maximized at $\theta=\theta_0$ for all $\alpha$.  Under the stated regularity conditions, this maximization is equivalent to the claim that, for all $\alpha \in (0,1)$, the first and second derivatives of $\psi_\alpha(\theta)$ at $\theta=\theta_0$ satisfy
\begin{align}
\psi_\alpha'(\theta_0) & = \int_{\XX(\alpha)} T_{\theta_0}(x) f_{\theta_0}(x) \,dx = 0, \label{eq:balance1} \\
\psi_\alpha''(\theta_0) & = \int_{\XX(\alpha)} V_{\theta_0}(x) f_{\theta_0}(x) \,dx < 0. \label{eq:balance2}
\end{align}
Since $T_{\theta_0}(X)$ has mean zero under $\prob_{X|\theta_0}$, we can see that \eqref{eq:balance1} requires $\XX(\alpha)$ to be somehow symmetric, or balanced, with respect to the distribution of $T_{\theta_0}(X)$.  We, therefore, refer to \eqref{eq:balance1} as the \emph{score-balance} condition.  This condition, expressed in terms of $\XX(\alpha)$ in \eqref{eq:balance1}, can be traced back to a corresponding condition on the predictive random set.  

Let us now assume that $\S_B$ is such that $\bel_X(\{\theta_0\}^c;\S_B) \sim \unif(0,1)$ under $\prob_{X|\theta_0}$; in the main text we construct a particular score-balanced predictive random set and show that that this assumption holds.  Then, as we argued above, for any $\alpha \in (0,1)$, there exists $t(\alpha) \in \TT$ such that $\bel_x(\{\theta_0\}^c; \S_B) \leq \alpha$ iff $T_{\theta_0}(x) \in B_{t(\alpha)}$.  In this case, for any $\theta$, 
\[ \prob_{X|\theta}\{\bel_X(\{\theta_0\}^c; \S_B) \leq \alpha\} = \int_{T_{\theta_0}(x) \in B_{t(\alpha)}} f_\theta(x) \,dx, \]
and the right-hand side is $\psi_\alpha(\theta)$ as defined previously.  From the definition of $B$, differentiating under the integral sign reveals that \eqref{eq:balance1} holds.  We can now prove 

\begin{prop}
\label{prop:balance1}
Focus on predictive random sets $\S$ such that $\bel_X(\{\theta_0\}^c; \S) \sim \unif(0,1)$ under $\prob_{X|\theta_0}$.  Then condition \eqref{eq:ordering} holds for all $\theta$ in a neighborhood of $\theta_0$ iff the predictive random set $\S=\S_B$ is score-balanced and 
\begin{equation}
\label{eq:balance2a}
\int_{T_{\theta_0}(x) \in B_t} V_{\theta_0}(x) f_{\theta_0}(x)\,dx < 0, \quad \forall\; t \in \TT.
\end{equation} 
\end{prop}

\begin{proof}[\indent Proof]
Take $\theta$ close enough to $\theta_0$ such that the remainder terms in a second-order Taylor approximation of $\psi_\alpha(\theta)$ about $\theta=\theta_0$ can be ignored.  That is, for any $\alpha$, 
\begin{align*}
\psi_\alpha(\theta)-\psi_\alpha(\theta_0) & = \int_{T_{\theta_0}(x) \in B_{t(\alpha)}} T_{\theta_0}(x) f_{\theta_0}(x) \,dx \cdot (\theta-\theta_0) \\
& \qquad + \frac12 \int_{T_{\theta_0}(x) \in B_{t(\alpha)}} V_{\theta_0}(x) f_{\theta_0}(x)\,dx \cdot (\theta-\theta_0)^2. 
\end{align*}
The first terms vanishes and the second term is negative by \eqref{eq:balance2a}.  Therefore $\psi_\alpha(\theta) < \psi_\alpha(\theta_0)$ for all $\alpha$ and, hence, \eqref{eq:ordering} holds for all $\theta$ in a neighborhood of $\theta_0$.  
\end{proof}

\bibliographystyle{/Users/rgmartin/Research/TexStuff/asa}
\bibliography{/Users/rgmartin/Research/mybib}

\pagebreak

\section{Corrections---added post-publication}
\label{S:correction}


\theoremstyle{plain} 
\newtheorem*{newthm1}{\indent Theorem~1$'$}
\newtheorem*{newthm3}{\indent Theorem~3$'$}
\newtheorem*{lem0}{\indent Lemma~0}

\subsection{Correction of Theorem~1}

In the main text, for validity of the predictive random set $\S$, the support $\SS$ was assumed only to be nested, i.e., for any $S,S' \in \SS$, either $S \subseteq S'$ or $S' \subseteq S$.  However, some additional technical conditions are required for the proof to go through.  

Fix a topology on the auxiliary variable space $\UU$, and let the $\sigma$-algebra defined there contain all the open sets.  In addition to being nested, we shall assume that $\SS$ contains both $\varnothing$ and $\UU$, and that all of its contents are closed subsets of $\UU$.  These additional requirements result in no real loss of generality.  Indeed, those predictive random sets in Corollary~1 of the main text already satisfy these.  These extra conditions also make the statement and proof of the theorem more transparent.   

\begin{newthm1}
Let $\SS$ be a nested collection of closed $\prob_U$-measurable subsets of $\UU$ that contains $\varnothing$ and $\UU$.  Define a predictive random set $\S$, with distribution $\prob_\S$, supported on $\SS$, such that 
\[ \prob_\S\{\S \subseteq K\} = \sup_{S \in \SS: S \subseteq K} \prob_U(S), \quad K \subseteq \UU. \]
Then $\S$ is valid in the sense of Definition~1 in the main text.  
\end{newthm1}

\begin{proof}
Set $Q(u)=\prob_\S\{\S \not\ni u\}$.  For any $\alpha \in (0,1)$, let $S_\alpha$ be the smallest $S \in \SS$ such that $\prob_\S\{\S \subseteq S\} \equiv \prob_U(S) \geq 1-\alpha$.  In particular, $S_\alpha = \bigcap \{S \in \SS: \prob_U(S) \geq 1-\alpha\}$.  Since each $S$ is closed, so is $S_\alpha$; it is also measurable by our assumptions about the richness of the $\sigma$-algebra on $\UU$.  The key observation is that $Q(u) > 1-\alpha$ iff $u \in S_\alpha^c$.  Therefore, by continuity of $\prob_U$ from above, we get 
\[ \prob_U\{Q(U) > 1-\alpha\} = \prob_U(S_\alpha^c) = 1-\prob_U(S_\alpha) = 1-\lim \prob_U(S), \]
where the limit is over all $S$ decreasing to $S_\alpha$.  By construction, each such $S$ satisfies $\prob_U(S) \geq 1-\alpha$.  So, finally, we get $\prob_U\{Q(U) > 1-\alpha\} \leq \alpha$ and, since $\alpha$ is arbitrary, the claimed validity is proved.  
\end{proof}

\subsection{Correction/extension of Theorem~3}  

Theorem~3 in the main text says that nested predictive random sets are more efficient than those which are not nested.  However, the nested predictive random set constructed in that theorem is not necessarily valid.  Since validity is a key to the IM analysis, it would be desirable if the new nested predictive random set $\S'$ was also valid.  We accomplish this in Theorem~3$'$ below.  First, we need the following lemma.

\begin{lem0}
On a space $\UU$ equipped with probability $\prob_U$, let $\S$ be a valid predictive random set for $U \sim \prob_U$.  Choose a collection of $\prob_U$-measurable subsets $\{\UU_x: x \in \XX\}$ of $\UU$, and set $\eta(x) = \prob_\S\{\S \subseteq \UU_x\}$.  Then 
\[ \inf_{x \in \XX_0} \eta(x) \leq \prob_U \Bigl\{ \bigcap_{x \in \XX_0} \UU_x \Bigr\} \]
for any subset $\XX_0$ of $\XX$ such that $\bigcap_{x \in \XX_0} \UU_x$ is $\prob_U$-measurable.  
\end{lem0}

\begin{proof}
First, note that if $u \in \UU_x^c$, then $Q(u) \equiv \prob_\S\{\S \not\ni u\} \geq \eta(x)$.  Therefore, if $u \in \bigcup_{x \in \XX_0} \UU_x^c$, then $Q(u) \geq \inf_{x \in \XX_0} \eta(x)$.  This argument implies
\[ \prob_U\Bigl\{ Q(U) \geq \inf_{x \in \XX_0} \eta(x) \Bigr\} \geq \prob_U \Bigl\{ \bigcup_{x \in \XX_0} \UU_x^c \Bigr\} = 1 - \prob_U\Bigl\{ \bigcap_{x \in \XX_0} \UU_x \Bigr\}. \]
Since $\S$ is valid, we have 
\[ \prob_U\Bigl\{ Q(U) \geq \inf_{x \in \XX_0} \eta(x) \Bigr\} \leq 1 - \inf_{x \in \XX_0} \eta(x); \]
Combining this with the inequality in the previous display, we get 
\[ 1-\inf_{x \in \XX_0} \eta(x) \geq 1 - \prob_U\Bigl\{ \bigcap_{x \in \XX_0} \UU_x \Bigr\}, \]
which implies $\inf_{x \in \XX_0} \eta(x) \leq \prob_U\{ \bigcap_{x \in \XX_0} \UU_x \}$.  
\end{proof}

A measurability question was overlooked in the main text.  In particular, the sets in \eqref{eq:intersect} below are not automatically measurable.  To confirm this, we shall add one more modification; note that this is not needed if the sampling model $\prob_{X|\theta}$ is discrete.  To start, for the given topology on $\UU$, keep the same assumptions about the corresponding $\sigma$-algebra as above.  Now, recall the a-events $\UU_x(A)=\{u \in \UU: \Theta_x(u) \subseteq A\}$ defined in the proof of Proposition~1 in the main text.  Here we shall replace $\UU_x(A)$ with its closure.  This does not affect any properties of the resulting belief function when $\prob_U$ is non-atomic.  In all the examples we have considered, $\prob_U$ can be taken as continuous; this is a particularly convenient choice, in light of Corollary~1 in the main text.  

\begin{newthm3}
Suppose that either $\XX$ is a discrete space, or that the assumptions in the previous paragraph hold.  Fix $A \subseteq \Theta$ and assume condition $(2.10)$ in the main text.  Given any valid predictive random set $\S$, there exists a nested and valid predictive random set $\S'$ such that $\bel_x(A;\S') \geq \bel_x(A;\S)$ for each $x \in \XX$.  
\end{newthm3}

\begin{proof}
For the given $A$ and $\S$, set $b(x) \equiv \bel_x(A;\S)$.  Define a collection $\SS' = \{S_x': x \in \XX\}$ of subsets of $\UU$ as follows:
\begin{equation}
\label{eq:intersect}
S_x' = \bigcap_{y \in \XX: b(y) \geq b(x)} \UU_y(A), 
\end{equation}
where $\UU_x(A)$ is the new closed a-event.  If necessary, add $\varnothing$ and $\UU$ to $\SS'$ to satisfy the requirement in Theorem~1$'$.  This collection $\SS'$ will serve as the support for the new predictive random set $\S'$.  First, we can see that $\SS'$ is nested: if $b(y) \geq b(x)$, then $S_y \supseteq S_x$.  Second, since the new a-events are closed, each $S_x'$ in \eqref{eq:intersect} is closed and, hence, $\prob_U$-measurable.  Third, define the measure $\prob_{\S'}$ for $\S'$ to satisfy
\[ \prob_{\S'}\{\S' \subseteq K\} = \sup_{x: S_x' \subseteq K} \prob_U(S_x'). \]
According to Theorem~1$'$, the new $\S'$ is valid.  Moreover, by Lemma~0 and the definition of $S_x'$, we have 
\begin{equation}
\label{eq:newthm3.bound}
\prob_U(S_x') \geq \inf_{y \in \XX: b(y) \geq b(x)} b(y) = b(x) \equiv \bel_x(A;\S). 
\end{equation}
Finally, we have a comparison of the belief functions corresponding to $\S$ and $\S'$:
\[ \bel_x(A;\S') = \prob_{\S'}\{\S' \subseteq \UU_x(A)\} \geq \prob_{\S'}\{\S' \subseteq S_x'\} = \prob_U(S_x') \geq \bel_x(A;\S); \]
the first inequality follows from monotonicity of $\prob_{\S'}\{\S' \subseteq \cdot\}$ and the fact that $S_x' \subseteq \UU_x(A)$ for each $x$, and the second inequality follows from \eqref{eq:newthm3.bound}.    
\end{proof}

\end{document}